\documentclass{amsart}

\usepackage{amssymb}
\usepackage{amsmath}
\usepackage{amsthm}
\usepackage{amsbsy}
\usepackage{bm}
\usepackage{hyperref}
\date{\today}
\usepackage{cite}
\usepackage{array}
\usepackage{color}

\def\and{\mbox{ and }}
\usepackage{tikz}
\usepackage{graphicx} 
\usepackage{cleveref}

\theoremstyle{theorem}
    \newtheorem{theorem}{Theorem}
    \newtheorem{lemma}{Lemma}
    \newtheorem{proposition}{Proposition}
    \newtheorem{corollary}{Corollary}

\theoremstyle{definition} % For roman text in the body
    \newtheorem{definition}{Definition}
    \newtheorem{fact}{Fact}
    \newtheorem{result}{Result}
    \newtheorem*{remark}{Remark}
    \newtheorem{example}[theorem]{Example}
    \newtheorem{exercise}[theorem]{Exercise}
    
    \newtheorem{assumption}{Assumption}

\def\suchthat{\; : \;}

\def\C{\mathbb{C}}

\def\N{\mathbb{N}}

\def\R{\mathbb{R}}

\def\l{\left}
\def\r{\right}
\def\<{\langle}
\def\>{\rangle}

\def\bar{\overline}
\def\P{{\bf P}}
\newcommand\mnote[1]{} %off
\newcommand\be{\begin{equation*}}

\newcommand\ee{\end{equation*}}

\newcommand\ben{\begin{equation}}
\newcommand\een{\end{equation}}
\newcommand\bes{\begin{eqnarray*}}
\newcommand\ees{\end{eqnarray*}}

\newcommand\bex{\begin{exercise}}
\newcommand\eex{\end{exercise}}
\newcommand\beg{\begin{example}}
\newcommand\eeg{\end{example}}
\newcommand\benu{\begin{enumerate}}
\newcommand\eenu{\end{enumerate}}
\newcommand\beit{\begin{itemize}}
\newcommand\eeit{\end{itemize}}
\newcommand\berk{\begin{remark}}
\newcommand\eerk{\end{remark}}
\newcommand\bdefn{\begin{defintion}}
\newcommand\edefn{\end{definition}}
\newcommand\bthm{\begin{theorem}}
\newcommand\ethm{\end{theorem}}
\newcommand\bprf{\begin{proof}}
\newcommand\eprf{\end{proof}}
\newcommand\blem{\begin{lemma}}
\newcommand\elem{\end{lemma}}

\newcommand{\sm}{{\raise0.3ex\hbox{$\scriptstyle \setminus$}}}

\def\l{\left}
\def\r{\right}

\def\CHI{\mathchoice%
{\raise2pt\hbox{$\chi$}}%
{\raise2pt\hbox{$\chi$}}%
{\raise1.3pt\hbox{$\scriptstyle\chi$}}%
{\raise0.8pt\hbox{$\scriptscriptstyle\chi$}}}
\def\smalloplus{\raise1pt\hbox{$\,\scriptstyle \oplus\;$}}

\theoremstyle{definition}

\title{Large deviations of crowding in finite $\beta$-ensembles }

\author{Kartick Adhikari}
\address{Department of Mathematics, Indian Institute of Science Education and Research, Bhopal 462066}

\email{kartickmath [at] iiserb.ac.in}

\author{Sitanath Majumder}
\address{Department of Mathematics, Indian Institute of Science Education and Research, Bhopal 462066}

\email{sitanath22 [at] iiserb.ac.in}

\date{\today}

\thanks{The research of KA was partially supported by the Inspire Faculty Fellowship: DST/InSPIRE/04/2020/000579.  KA would also like to acknowledge support from ICTP through the Associates Programme (2026-2032)}

\keywords{Random matrix, variance profile, limiting spectral distribution, covariance matrix.}

\begin{document}

\begin{abstract}
 We consider finite $\beta$-ensembles $\mathcal X_{n,\beta}^{\mathbb F}$ with $n$ points on $\mathbb F$,  where $\mathbb F$ denotes either the real line or the complex plane. Let $U$ be a bounded subset of $ \mathbb F$ such that $\partial U$ (the boundary of $U$) is polar for $\mathbb F=\R$ and $\partial U$
 is a closed $1$--rectifiable set with finite $1$-dimensional Hausdorff measure. Suppose $\mathcal X_{n,\beta}^{\mathbb F}(U)$ denotes the number of points in the region $U$. We show that the sequence of laws of $\{n^{-1}\mathcal X_{n,\beta}^{\mathbb F}(U); n\ge 1\}$ satisfies the large deviation type bound  with speed $n^2$ and with a good rate function. For $\mathbb{F} = \mathbb{R}$, this result can be derived using the contraction principle. However, when $\mathbb{F} = \mathbb{C}$, the contraction principle does not yield the desired outcome. Therefore, we adopt a direct approach to establish our results.
 \end{abstract}

\maketitle

\section{Introduction and main results}
 \subsection{The large deviation principle} The theory of {Large Deviation Principle} (LDP) is a fundamental area in probability theory that studies the asymptotic behavior of rare events, providing precise exponential estimates for their probabilities. The origins of large deviation theory lie in the work of Boltzmann on entropy in 1870 and \textit{Harald Cramér} (1938)\cite{cramer1994nouveau}, who analyzed the decay rates of tail probabilities for sums of independent random variables. Subsequent foundational contributions by \textit{Donsker and Varadhan} (1975) and \textit{Freidlin and Wentzell} (1984) led to a comprehensive framework for large deviation principles in various probabilistic settings. This indicates that a unifying mathematical formalism was only developed, starting with Varadhan's definition of a 'large deviation principle' in 1966\cite{varadhan1966asymptotic}. The most important basic results of the abstract
theory were proved more or less between 1966 and 1991, when O'Brien en
Verwaat \cite{o1991capacities} and Puhalskii \cite{pukhalskii1994theory} proved that exponential tightness implies
a subsequential LDP. The classical books on the topic are those of Deuschel and Stroock \cite{deuschel2001large},\cite{deuschel1989large}
and especially Dembo and Zeitouni \cite{dembo2009large}.

The \textit{large deviation principle} (LDP) provides a characterization of the limiting behavior of a sequence of probability measures $\{\mu_n\}$ defined on $(\mathcal{X}, \mathcal{B})$, where $\mathcal X$ is a metric space and $\mathcal B$ denotes the Borel $\sigma$-field on it. This principle is formulated using a \textit{rate function}, which gives exponential bounds for the measures assigned to specific subsets. A \textit{rate function} is defined as a lower semicontinuous function $I: \mathcal{X} \to [0, \infty]$ that satisfies the property that for any $\alpha \in [0, \infty)$, the level set $\Psi_I(\alpha) = \{ x \in \mathcal{X} \mid I(x) \leq \alpha \}$ is a closed subset of $\mathcal{X}$. A \textit{good} rate function is one where all level sets $\Psi_I(\alpha)$ form compact subsets of $\mathcal{X}$. 
The \textit{effective domain} of $I$, denoted as $D_I$, consists of all points in $\mathcal{X}$ where the rate function is finite, namely, 
\[
D_I = \{ x \in \mathcal{X} \mid I(x) < \infty \}.
\]
When context is clear, we simply refer to $D_I$ as the \textit{domain} of $I$.

%\noindent If $\mathcal{X}$ is equipped with a metric, then the lower semicontinuity condition can be examined through sequences. Specifically, $I$ is lower semicontinuous if and only if for every sequence $\{x_n\} \subset \mathcal{X}$ converging to $x$, the following holds:
%\[
%\liminf_{n \to \infty} I(x_n) \geq I(x).
%\]
%An immediate consequence of this definition is that the infimum of a rate function is always attained over closed sets.\\
\begin{definition}[The large deviation principle]
	Let $a_n$ be a monotonic increasing sequence with $\lim_{n\to \infty} a_n=\infty.$
A sequence of probability measures $\{\mu_n\}$ on $(\mathcal X,\mathcal B)$ is said to satisfy a \textit{large deviation principle} (LDP) with speed $a_n$ and  rate $I$ if the following hold
\begin{enumerate}
	\item for any open $\mathcal O\subseteq \mathcal X$,
	\[
	\liminf_{n \to \infty}a_n^{-1}\log \mu_n(\mathcal O)\ge - \inf_{x\in \mathcal O}I(x),
	\]
	
	\item for any closed $\mathcal F\subseteq \mathcal X$,
	\[
	\liminf_{n \to \infty}a_n^{-1}\log \mu_n(\mathcal F)\le - \inf_{x\in \mathcal F}I(x).
	\]
\end{enumerate}
\end{definition}

The LDP has been studied extensively in random matrix theory, in particular for various spectral statistics, including empirical spectral measures and extreme eigenvalues. In  \cite{arous1997large},  Arous and Guionnet established the large deviation for Wigner’s semicircle law. Arous and Zeitouni proved the large deviation for the Ginibre ensemble \cite{ben1998large}.
In the case of Gaussian ensembles, Coulomb gas methods have been successfully employed to derive precise rate functions for such deviations \cite{majumdar2009large} \cite{MR2926763}. For non-Gaussian Wigner matrices, recent breakthroughs have extended LDPs to matrices with sub-Gaussian and heavy-tailed entries, revealing the role of extreme matrix elements in deriving large deviations \cite{MR4235485} \cite{cook2023full} \cite{MR3265172}.

Many recent developments have been made in random matrix theory, which includes LDP for Wigner matrices without Gaussian tails by Bordenave and Caputo \cite{MR3265172},
the LDP for the largest eigenvalue of Wigner matrices without Gaussian tails \cite{MR3492936}, and the intermediate deviation regime for the full eigenvalue statistics in the complex Ginibre ensemble \cite{lacroix2019intermediate}. The LDP for the largest eigenvalue of matrices with variance profiles \cite{husson2022large}. In this article, we study the LDP type bound for the crowding of finite $\beta$-ensembles on the real line and the complex plane. 

\subsection{The finite $\beta$-ensembles}
Throughout, the set $\mathbb F$ represents either the set of real numbers or the set complex numbers. In other words  $\mathbb{F}=\mathbb{R}$ or $\mathbb{F}=\mathbb{C}$, where $\R$ and $\C$ denote the set of real and complex numbers respectively. Suppose $m$ denotes the Lebesgue measure on $\mathbb F$.
We consider the finite $\beta$-ensembles on $\mathbb F$, where $\beta>0$.   Let $\mathcal X_{n,\beta}^{\mathbb F}$ denote the finite $\beta$-ensemble with $n$ points on $\mathbb F$. The joint distribution of $n$ points $(\lambda_1,\ldots,\lambda_n)$  is defined by 
\begin{align*}
	\P^{\mathbb F}_{n,\beta}(d\lambda_1, \ldots, d\lambda_n) = \frac{1}{Z^{\mathbb F}_{n,\beta}} \prod_{1 \leq i < j \leq n} |\lambda_i - \lambda_j|^\beta e^{-n\sum_{i=1}^n V(\lambda_i)} \prod_{i=1}^n dm(\lambda_i),
\end{align*}
where $Z^{\mathbb F}_{n,\beta}$ denotes the normalizing constant, also known as {\it the partition function},
\begin{align}\label{eqn:Z}
Z^{\mathbb F}_{n,\beta}= \int_{\mathbb{F}} \cdots \int_{\mathbb{F}} |\Delta(\lambda)|^\beta e^{-n\sum_{i=1}^n V(\lambda_i)} \prod_{i=1}^n dm(\lambda_i),
\end{align}
and $V$ satisfies Assumption \ref{ass:V}. 
\begin{assumption}\label{ass:V} We assume that $V: \mathbb{F} \to \mathbb{R}_+$  is a continuous function which satisfies the following condition, for some $\beta' > 1$ satisfying $\beta' > \beta$,
	\[
	\liminf_{|x| \to \infty} \frac{V(x)}{\beta' \log |x|} > 1.
	\]
	In addition, for  $\mathbb{F}=\mathbb{C}$, we assume that $V\in C^{1,1}_{\mathrm{loc}}(\mathbb{C})$.
\end{assumption}

 In other words, the external potential has at least logarithmic growth at  infinity,  and it is continuously differentiable and its first derivative is Lipschitz continuous. Throughout, we assume $V$ satisfies Assumption \ref{ass:V}. Thus, for ease of writing, we suppress the dependence of $V$ in the notation.

This model has been studied in both statistical physics and random matrix theory extensively \cite{forrester2010log}\cite{MR3262506}. In statistical physics, when $\mathbb F=\mathbb R$, this model is used to explain the one-dimensional log-gas system \cite{dyson1962statistical},\cite{serfaty2024lectures}. In this case, the points represent the positions of the particles; the Vandermonde term captures the pairwise repulsion; $\beta$ represents the inverse temperature that controls the strength of the logarithmic repulsion; and $V$ corresponds to the external potential in the system. Whereas, when $\mathbb{F}=\mathbb C$ this model is used to explain the two-dimensional Coulomb gas system \cite{MR3353821}, \cite{hedenmalm2013coulomb},\cite{ben1998large}.

On the other hand, the $\beta$-ensembles can be seen as a generalisation of the eigenvalue distribution of random matrices. In particular, the ensembles  $\mathcal X_{n,\beta}^{\mathbb R}$ for $\beta=1,2,4$ with $V$ quadratic coincide with the joint distribution of the eigenvalues of the GOE (Gaussian orthogonal ensemble), GUE (Gaussian unitary ensemble), GSE (Gaussian simplectic ensemble) ensembles, respectively. The GOE, GUE, and GSE correspond to the $n$ eigenvalues of the real symmetric matrices, complex Hermitian matrices, and self-adjoint quaternion matrices, respectively. In the complex case, the ensemble $\mathcal X_{n,2}^{\mathbb C}$ with $V(x)=x^2/2$ coincides with the joint eigenvalues of an $n\times n$  random matrix with i.i.d. complex Gaussian entries with mean zero and variance $1/n$, known as the $n$-th Ginibre ensemble \cite{ginibre1965statistical}.

Large deviation for the $\beta$-ensemble was first introduced by Ben Arous and Alice Guionnet in their paper \cite{arous1997large} for Wigner matrices  and later extended to invariant ensembles with general potentials. A systematic treatment of the LDP for $\beta$-ensembles with general confining potentials in $\mathbb R$ can be found in \cite{anderson2010introduction}. The extension to the general values of $\beta>0$ was defined through the tridiagonal models introduced by Dumitriu and Edelman in \cite{MR1936554}. Further developments, including the weaker assumptions on the potential and connections with stochastic calculus, can be found in \cite{MR2095566}.
 In this article, we study the LDP of crowding of $\mathcal X_{n,\beta}^{\mathbb F}$ in a given domain. More precisely, suppose $U\subset \mathbb F$  and $\mathcal X_{n,\beta}^{\mathbb F}(U)$ denotes the number of points that lies in $U$. We study the LDP type bound of the laws of $n^{-1}\mathcal X_{n,\beta}^{\mathbb F}(U)$  for a class of $U$. Before stating the result, we introduce a few notations and definitions.  
 
  Let $(\lambda_1,\ldots, \lambda_n)$ be a collection of $\mathbb{F}$-valued random variables with their joint distribution $\P_{n,\beta}^{\mathbb F}$. Define  $X_k:=\delta_{\lambda_k}(U)$, for $k=1,\ldots, n$, and 
 \begin{align}\label{eqn:defncrowding}
\frac{1}{n}\mathcal X_{n,\beta}^{\mathbb F}(U)=\frac{1}{n}\sum_{k=1}^{n}X_k.
 \end{align}
 Observe that all the random variables $X_1,\ldots, X_n$ have the same distribution but not independent. Thus our problem is analogous to the problem of Cramer's theorem with dependent random variables.
 %%%%%%%%%%%%%%%%%%%%%%%%%%%%%%%%%%%%%%%%%%%%%%%%%%%%%%%%%%%%%%%%%%%%%%%%%%%%%% Add one paragraph about the Real ginibre and the complex ginibre ensemble. Also add one paragraph to state all the LDP reslts related that model.%%%%%%%%%%%%%%%%%%%%%%%%%%%%%%%%%%%%%%%%%%%%%%%%%%%%%%%%%%%%%%%%%%%%%%%%%%%%%%%%%%%%%%%%%%%%%%%%%%%%%%%%%%%%%%%%%%%%%%%%%%%%%%%%%%%%%%%%%%%%%%%%%%%%%%%%%%%%%%%%%%%%%%%%%%%%%%%%%%%%%%%%%%%%%%%%%%%%%%%%%%%%%%%%%%%%%%%%%%%%%%%%%
%\begin{definition}

Let $M_1(\mathbb F) $ denotes the set of all probability measures on $\mathbb F$.
%equipped with the usual weak topology on it.
Define the \textit{non-commutative entropy} $\Sigma : M_1(\mathbb{F}) \to [-\infty, \infty)$ as
\[
\Sigma(\mu): = 
\begin{cases}
\iint \log|x-y| d\mu(x) d\mu(y) & \text{if } \int \log(|x|+1) d\mu(x) < \infty, \\
-\infty & \text{otherwise}.
\end{cases}
\]
%\end{definition}
%\begin{definition}
Define the function $I_\beta : M_1(\mathbb{F}) \to [0,\infty]$ as
\begin{align}\label{eqn:defIbeta}
I_\beta(\mu) := 
\begin{cases}
\int V(x) d\mu(x) - \frac{\beta}{2} \Sigma(\mu) - c_\beta & \text{if } \int V(x) d\mu(x) < \infty, \\
\infty & \text{otherwise},
\end{cases} 
\end{align}
%\end{definition}
where
\(
c_\beta = \inf_{\mu \in M_1(\mathbb{F})} \left\{ \int V(x) d\mu(x) - \frac{\beta}{2} \Sigma(\mu) \right\}. 
\) 
Suppose $A\subset \mathbb F$, for $\epsilon>0$, define 
\[
A_\epsilon=\cup_{x\in A}B(x, \epsilon), \mbox{ where } B(x,\epsilon)=\{y\in \mathbb F\suchthat \|x-y\|<r\},
\]
 the $\epsilon$-neighbourhood of $A$, $\|\cdot\|$ denotes the usual norm on $\mathbb F$.  We define
 \begin{align}\label{eqn:defm1F}
 	M_1^*(\mathbb F)=\l\{ \begin{array}{ll}
 		\{\mu\in M_1(\mathbb R)\suchthat \mu(\partial U)=0 \} & \mbox{ if $\mathbb F=\R$},\\\\
 		\{\mu\in M_1(\mathbb C) \suchthat \mu((\partial U)_\epsilon)\le c\epsilon, \mbox{ for } 0<\epsilon\le \delta_0\} & \mbox{ if $\mathbb F=\C$},
 	\end{array}\r.
 \end{align}
% \begin{align*}
%  M_1^*(\mathbb C)=\{\mu\in M_1(\mathbb C) \suchthat \mu((\partial U)_\epsilon)\le c\epsilon, \mbox{ for } 0<\epsilon\le \delta_0\},
% \end{align*} 
 where $c$ and $\delta_0$ are two positive constants (here we take $c=2\mathcal H^1(\partial U)$, where $\mathcal H^1(\partial U)$ is 1-dimensional Hausdorff measure on $\partial{U}$, by assuming regularity condition on U, reason shown in \eqref{cor:tube_estimate_rectifiable}).
% For $\mathbb F=\mathbb R$ we define $M_1^*(\mathbb R)$ as 
% \begin{align*}
% 	M_1^*(\mathbb R)=\{\mu\in \mathbb R|\mu(\partial U)=0 \}.
% \end{align*}
 We define $\gamma:[0,1]\to [0,\infty]$ as   
\begin{align}\label{eqn:defgamma}
	\gamma(x):=\inf\{I_{\beta}(\mu) \suchthat \mu(U)=x, \mu\in M_1^*(\mathbb{F}) \}, \mbox{ for all $x\in [0,1]$}.
\end{align}
Throughout $\partial U$ denotes the boundary of $U$. The following regularity condition on $\partial U$ will be assumed in our result.
\begin{assumption}[Regularity condition on $U$]\label{ass:U} Let $U\subset \mathbb F$. We assume that 
 \begin{enumerate}
 	\item if $\mathbb{F}=\mathbb{R}$ then $\partial U$ is a polar set, and
 	\item if $\mathbb{F}=\mathbb{C}$ then $U\subset\mathbb{C}$ be a bounded domain whose boundary $\partial U$
 	is a closed $1$--rectifiable set and satisfies
 	\(
 	\mathcal H^1(\partial U)<\infty,
 	\)
 	where $\mathcal H^1(\cdot)$ denotes the $1$-dimensional Hausdorff measure.
 \end{enumerate}
\end{assumption}
\noindent We refer to Section \ref{sec:pre} for the definitions of a $1$-rectifiable set and the $1$-dimensional Hausdorff measure. Now we are ready to state our main result.
\begin{theorem}\label{theorem2}
Let $U\subset \mathbb F$ which satisfies Assumption \ref{ass:U}, and $n^{-1}\mathcal X_{n,\beta}^{\mathbb F}(U) $ be as defined in \eqref{eqn:defncrowding}.    Then the sequence of laws of $\{n^{-1}\mathcal X_{n,\beta}^{\mathbb F}(U) \}$ satisfies the following large deviation type bounds with speed $n^2$ and good rate  $\gamma$ as  defined in \eqref{eqn:defgamma}.
\begin{enumerate}
\item[(i)] For any open set $O \subseteq [0,1]$,
\[
\liminf_{n \to \infty} \frac{1}{n^2} \log \P^{\mathbb F}_{n,\beta} (n^{-1}\mathcal X_{n,\beta}^{\mathbb F}(U)  \in O) \geq -\inf_{x\in O} \gamma(x),
\]
\item[(ii)] For $\mathbb F= \mathbb{R}$ and any closed set $F \subseteq [0,1]$,
\[
\limsup_{n \to \infty} \frac{1}{n^2} \log \P^{\mathbb R}_{n,\beta}(n^{-1}\mathcal X_{n,\beta}^{\mathbb R}(U)  \in F) \leq -\inf_{x\in F} \gamma(x).
\]
However, for $\mathbb F=\mathbb C$, the last inequality holds if $F$ satisfies
\begin{align}\label{eqn:condiononF}
\P_{n,\beta}^{\mathbb{C}}(L_n\in \mathcal S_F\cap M_1^*(\mathbb C)^c)\le e^{n^\alpha}\P_{n,\beta}^{\mathbb{C}}(L_n\in \mathcal S_F\cap M_1^*(\mathbb C)),
\end{align}
for some $0\le \alpha<2$ and for large $n$, where $\mathcal S_F=\{\mu\in M_1(\C)\suchthat \mu(U)\in F\}$.
\end{enumerate}
\end{theorem}

%Note that if $\mu_V(U)\in F$, that is, $\mu_V\in \mathcal S_F$ then \eqref{eqn:condiononF} holds trivially. 

\begin{remark}  
	\begin{enumerate} 
		\item  \Cref{theorem2} states that the full LDP holds when $\mathbb F=\mathbb R$. However, the full LDP does not hold for $\mathbb F=\mathbb C$.
\item Suppose $\mu_V$ is the weighted equillibrium measure with externel potential $V/2$. It can be shown that $\mu_V\in M_1^*(\C)$. See \Cref{cor:equillibriummeasure}. It is not difficult to see that if $\mu_V(U)\in F$, that is, $\mu_V\in \mathcal S_F$ then \eqref{eqn:condiononF} holds.
		
		\item Instead of considering the law of $n^{-1}\mathcal X_{n,\beta}^{\mathbb F}(U)$, if we consider the family of laws say $P_{n,\beta}^{\mathbb{C}}(L_n\in M_1^*(C)\cap \mathcal S_F)$, the full large deviation bound will hold directly. See \Cref{re:contraction principle} for the contraction principle. In our case the full large deviation will hold for $\mathbb{F}=\mathbb{R}$ which can be derived from the following proposition and the contraction principle. However, for the complex case the LDP does follow from the contraction principle. In this case,  the lower bound follows easily. However, we do not have the large deviation type upper bound without the condition \eqref{eqn:condiononF} on the closed set $F$. 
	\end{enumerate}

\end{remark}

\begin{result}\label{re:contraction principle}\cite[Theorem 4.2.1]{dembo2009large}
	Let $\mathcal{X}$ and $\mathcal{Y}$ be Hausdorff topological spaces and
	$f:\mathcal{X}\to\mathcal{Y}$ a continuous function. Consider a good
	rate function $I:\mathcal{X}\to[0,\infty]$.
	
	\begin{enumerate}
		\item[(a)] For each $y\in\mathcal{Y}$, define
		\[
		I'(y)= \inf\{I(x):x\in\mathcal{X},\ y=f(x)\}.
		\]
		 Then $I'$ is a good rate function on $\mathcal{Y}$, where the
		infimum over the empty set is taken as $\infty$.
		
		\item[(b)] If $I$ controls the LDP associated with a family of probability
		measures $\{\mu_n\}$ on $\mathcal{X}$, then $I'$ controls the
		LDP associated with the family of probability measures
		$\{\mu_n\circ f^{-1}\}$ on $\mathcal{Y}$.
	\end{enumerate}
\end{result}

The proof of \Cref{theorem2} is given in \Cref{sec:proofofthm}.  The key result for the proof is the following proposition.   The  empirical measure of $n$ points $\lambda_1,\ldots, \lambda_n\in \mathbb F$ is denoted by $L_n$, and is defined as
\begin{align}\label{eqn:defLn}
L_n := \frac{1}{n} \sum_{k=1}^n \delta_{\lambda_k},
\end{align}
where $\delta_{\lambda}$ denotes the Dirac measure on $\mathbb F$. 
\begin{proposition}\label{proposition}
	Let $(\lambda_1,\ldots,\lambda_n)$  be distributed according to the law $\P_{n,\beta}^{\mathbb F}$, and $L_n$ be as defined in \eqref{eqn:defLn}.  Then the family of random measures $\{L_n\}$ satisfies  the large deviation principle in $M_1^*(\mathbb{F})$ equipped with the weak topology  with speed $n^2$ and good rate  $I_\beta$ as defined in \eqref{eqn:defIbeta}.
	In  other words,
	\begin{enumerate}
		\item for any open set $\mathcal O \subset M_1^*(\mathbb{F})$,
		\[
		\liminf_{n \to \infty} \frac{1}{n^2} \log \P^{\mathbb F}_{n,\beta} (L_n \in \mathcal O) \geq -\inf_{\mu\in \mathcal O} I_\beta(\mu),
		\]
		\item for any closed set $\mathcal F \subset M_1^*(\mathbb{F})$,
		\[
		\limsup_{n \to \infty} \frac{1}{n^2} \log \P^{\mathbb F}_{n,\beta}(L_n \in \mathcal F) \leq -\inf_{\mu\in \mathcal F} I_\beta(\mu).
		\]
	\end{enumerate}
	\end{proposition}
\begin{remark}
The reason that we choose to work with $M_1^*(\mathbb{C})$ for the complex case is that it is a closed subset of $M_1(\mathbb{C})$ and that helps us to show that $\gamma$ is a good rate function. We could have worked with $\{\mu\in M_1(\mathbb{C})|\mu(\partial{U})=0\}$ to directly use the contraction principle and to derive our result, but then proving that $\gamma$ a good rate function would be difficult as this space is not closed in $M_1(\mathbb{C})$. 
%In fact, it is not clear whether $I_\beta$ is good on $\{\mu\in M_1(\mathbb{C})|\mu(\partial{U})=0\}$.
\end{remark}
  The proofs of \Cref{proposition} for $\mathbb F=\R$ and $\mathbb F=\C$  are given in \Cref{sec:proofofprop} and \Cref{sec:proof} respectively. In \cite{arous1997large}, Arous and Guionnet showed that the laws of the random measures $\{L_n\}$ satisfy the LDP  when $\mathbb F=\mathbb R$ in $M_1(\R)$ equipped with the weak topology  with speed $n^2$ and good rate function $I_\beta$.  On the other hand, Arous and Zeitouni  \cite{ben1998large} proved that $\{L_n\}$ satisfies the LDP in $M_1^S(\mathbb C)$ (the space of all symmetric measures equipped with weak topology) with rate $n^2$ and with rate function $I_\beta$ when $\mathbb F=\mathbb C$ and $\beta=2$. In \cite{MR3262506}, Chafai, Gozlan and Zitt showed that the LDP of $\{L_n\}$ holds in $ M_1(\mathbb C)$ for general $\beta$. We derived \Cref{proposition}  for $\mathbb F=\mathbb R$ from the LDP result of $\{L_n\}$ in $M_1(\mathbb R)$.  See Section \ref{sec:proofofprop}. However, this method does not work in the case of $\mathbb F=\mathbb C$. Thus, in this case, we provide a direct proof of the proposition, inspired by the proofs in \cite{ben1998large} and \cite{MR3262506}. See Section \ref{sec:proof}.

\subsection{Related work and motivation} 
The LDP for outlying coordinates in $\beta$-ensembles was established by Bloom in \cite{bloom2014large}. In \cite{Holcomb2017OvercrowdingAF}, Holcomb and Valkó estimated the overcrowding probability for the Sine$_\beta$ process, which serves as the bulk limiting point process of the Gaussian $\beta$-ensemble. In \cite{krishnapur2006overcrowding}, Krishnapur estimated the overcrowding for the zeros of planar and hyperbolic Gaussian analytic functions.

 In \cite{shirai2015ginibre}, Shirai examined the infinite Ginibre ensemble $\mathcal{X}_\infty$, which is the limiting point process of the finite Ginibre ensemble. Alternatively, the process $\mathcal{X}_\infty$ is described as a determinantal point process with the kernel $e^{z\bar{w}}$ with respect to the standard complex Gaussian measure. Let $B(0,r)$ denote the ball of radius $r$ centered at the origin, and let $\mathcal{X}_\infty(B(0,r))$ represent the number of points within $B(0,r)$. In this work, Shirai established the LDP for the random variables $\{r^{-2} \mathcal{X}_\infty(B(0,r)) \,|\, r>0\}$. Our problem is partially inspired by this problem.

 The rest of the article is organized as follows. In the next section, we recall some preliminaries and derive some elementary results that will be essential for proving our main result. In \Cref{sec:proofofthm}, we provide the proof of \Cref{theorem2} assuming the proposition. We prove \Cref{proposition} for $\mathbb F=\mathbb R$ using the known LDP results in \Cref{sec:proofofprop}. The proof of \Cref{proposition}  for $\mathbb F=\mathbb C$ is provided in \Cref{sec:proof}.
 
%%%%%%%%%%%%%%%%%%%%%%%%%%%%%%%%%%%%

\section{Preliminaries and basic results}\label{sec:pre}
  In this section, we recall some basic definitions and derive some elementary results that will be used in the proof of our main result. We mainly explain the regularity conditions related to the external potential \( V \) and the boundary \( \partial U \).

\subsection{Minkowski content and tube estimates in $\mathbb{R}^2$}

Let \( m \) denote the \( d \)-dimensional Lebesgue measure on \( \mathbb{R}^d \). The notation \( m \) should be interpreted as indicated by the context.
\begin{definition}[Upper and lower Minkowski content]
	Let $A\subset\mathbb{R}^n$ and let $0\le d\le n$ be an integer. Denote
	\(
	A_r:=\{x\in\mathbb{R}^n:\operatorname{dist}(x,A)<r\}
	\), the $r$-neighbourhood of $A$.
	The $d$--dimensional upper Minkowski content of $A$ is defined as
	\[
	\mathcal M^{*d}(A)
	=
	\limsup_{r\downarrow 0}
	\frac{m(A_r)}{\alpha(n-d)\,r^{\,n-d}},
	\]
	and the $d$--dimensional lower Minkowski content is defined as
	\[
	\mathcal M_{*}^{d}(A)
	=
	\liminf_{r\downarrow 0}
	\frac{m(A_r)}{\alpha(n-d)\,r^{\,n-d}},
	\]
	where $\alpha(k)$ denotes the Lebesgue measure of the unit ball in
	$\mathbb{R}^k$. If $	\mathcal M^{*d}(A)=\mathcal M_{*}^{d}(A)$, then the common value is called
	the $d$--dimensional Minkowski content of $A$ and is denoted by
	$\mathcal M^{d}(A)$.
\end{definition}
 In particular, for the planar case $n=2$ and $d=1$, we have $\alpha(1)=2$, hence
\[
\mathcal M^{*1}(A)
=
\limsup_{r\downarrow 0}
\frac{m(A_r)}{2r},
\qquad
\mathcal M_{*}^{1}(A)
=
\liminf_{r\downarrow 0}
\frac{m(A_r)}{2r}.
\]

	\begin{definition}[$d$-dimensional Hausdorff measure]
		Let $E \subset \mathbb{R}^n$ and let $d \ge 0$. For $\delta > 0$ define
		\[
		\mathcal{H}^d_\delta(E)
		=
		\inf \left\{
		\sum_{i=1}^{\infty} (\operatorname{diam} U_i)^d
		:\;
		E \subset \bigcup_{i=1}^{\infty} U_i, \;
		\operatorname{diam}(U_i) < \delta
		\right\},
		\]
		where $\operatorname{diam}(U) := \sup\{\|x-y\| : x,y \in U\}$.	The \emph{$d$-dimensional Hausdorff measure} of $E$ is defined by
		\[
		\mathcal{H}^d(E)
		=
		\lim_{\delta \to 0} \mathcal{H}^d_\delta(E)
		=
		\sup_{\delta>0} \mathcal{H}^d_\delta(E).
		\]
	\end{definition}
	\begin{definition}[a $d$-rectifiable set]
		A $d$-rectifiable set $E\subset \R^n$  is a subset with finite 
		$d$-dimensional Hausdorff measure, that is $\mathcal{H}^d(E)<\infty$, which is covered (up to a set of measure zero) by a countable union of images of Lipschitz functions defined on $\R^d$.
	\end{definition}
	
\begin{result}\cite[p. 275, Theorem 3.2.39]{federer2014geometric}
	\label{thm:minkowski_equals_hausdorff}
	Let $E\subset\mathbb{R}^n$ be a closed $d$--rectifiable set. Then the $d$--dimensional Minkowski content of $E$ exists and equals to the $d$-dimensional Hausdorff measure, that is,
	\[
	\mathcal M^{d}(E)
	=
	\mathcal H^{d}(E).
	\]
	%where $\mathcal H^m$ denotes the $m$-dimensional Hausdorff measure.
\end{result}

\noindent In particular we have the following corollary  which will be used in this article.

\begin{corollary}[Tube estimate in $\mathbb{R}^2$]
	\label{cor:tube_estimate_rectifiable}
	Let $U\subset\mathbb{R}^2$ be a bounded domain whose boundary $\partial U$
	is a closed $1$--rectifiable set with
	\(
	\mathcal H^1(\partial U)<\infty.
\)
	Then
	\[
	\lim_{r\downarrow 0}\frac{m\big((\partial U)_r\big)}{2r}
	=
	\mathcal H^1(\partial U).
	\]
	In particular, there exists $r_0>0$ and a constant $C_1>0$ such that for all
	$0<r<r_0$,
	\[
	m\big((\partial U)_r\big)
	\le
	C_1\,r.
	\]
\end{corollary}

\begin{proof}
	Since $\partial U$ is  a closed $1$--rectifiable subset of
	$\mathbb{R}^2$, for $n=2$, $d=1$, Result \ref{thm:minkowski_equals_hausdorff} implies that 
 the $1$--dimensional Minkowski content of $\partial U$ exists and
	\[
	\lim_{r\downarrow 0}
	\frac{m\big((\partial U)_r\big)}{2r}
	=
	\mathcal H^1(\partial U)
	<\infty.
	\]
	Hence there exists $r_0>0$ such that for all $0<r<r_0$,
	\[
	\frac{m\big((\partial U)_r\big)}{2r}
	\le
	2\,\mathcal H^1(\partial U).
	\]
	Multiplying both sides by $2r$ yields
	\(
	m\big((\partial U)_r\big)
	\le
	C r,
	\)
	which gives the desired bound with
	\(
	C = 4\,\mathcal H^1(\partial U).
	\)
\end{proof}

\begin{remark}
	The above applies in particular when $\partial U$ is the image of a bounded
	interval under a Lipschitz parametrization, for instance for piecewise $C^1$
	or Lipschitz planar curves. In this case the $1$--dimensional Hausdorff
	measure $\mathcal H^1(\partial U)$ coincides with the usual arc length, and
	the Minkowski content description yields the linear growth of the area of
	the tubular neighbourhood:
	\[
	m\big((\partial U)_r\big)
	\sim
	2r\,\mathcal H^1(\partial U)
	\quad\text{as } r\downarrow 0.
	\]
\end{remark}
\subsection{The space $C^{1,1}_{\mathrm{loc}}(\mathbb{R}^2)$}

 We recall some basic definitions for further discussion. Let $Q:\mathbb{R}^2 \to \mathbb{R}$ be a function.

\begin{definition}[Class $C^1(\mathbb{R}^2)$]
	A map $Q$  is said to be a continuously differentiable function on $\R^2$ if the first partial derivatives
	\(
	\frac{\partial Q}{\partial x_1} \mbox{ and }
	\frac{\partial Q}{\partial x_2}
\)
	exist at every point of $\mathbb{R}^2$ and are continuous functions on $\mathbb{R}^2$. The set  of all continuously differentiable functions on $\R^2$ are denoted by $C^1(\mathbb{R}^2)$.
\end{definition}

\begin{definition}[Gradient]
	The gradient of $Q$ is the vector-valued function on the plane
	\(
	\nabla Q : \mathbb{R}^2 \to \mathbb{R}^2
	\)
	defined by
	\[
	\nabla Q(x)
	=
	\left(
	\frac{\partial Q}{\partial x_1}(x),
	\frac{\partial Q}{\partial x_2}(x)
	\right),
	\qquad x=(x_1,x_2)\in\mathbb{R}^2 .
	\]
\end{definition}

\begin{definition}[Locally Lipschitz function]
	A function $F:\mathbb{R}^2 \to \mathbb{R}^m$ is said to be \emph{locally Lipschitz}
	if for every compact set $K \subset \mathbb{R}^2$ there exists
	$L_K >0$ such that
	\[
	|F(x)-F(y)| \le L_K \|x-y\|,
	\quad
	\text{for all } x,y \in K .
	\]
\end{definition}

\begin{definition}[The space $C^{1,1}_{\mathrm{loc}}(\mathbb{R}^2)$]
	We say that a function $Q:\mathbb{R}^2 \to \mathbb{R}$ belongs to
	$C^{1,1}_{\mathrm{loc}}(\mathbb{R}^2)$ if
	\(Q \in C^1(\mathbb{R}^2)\)
	and its gradient $\nabla Q$ is locally Lipschitz on $\mathbb{R}^2$.
\end{definition}

%\subsection{Second derivatives/Hessian and Laplacian}

\begin{definition}[Second derivative / Hessian matrix]
	If the second partial derivatives exist, the \emph{Hessian matrix} of $Q$
	is the matrix-valued function
	
	\[
	D^2 Q(x)
	=
	\begin{pmatrix}
		\frac{\partial^2 Q}{\partial x_1^2}(x)
		&
		\frac{\partial^2 Q}{\partial x_1 \partial x_2}(x)
		\\[6pt]
		\frac{\partial^2 Q}{\partial x_2 \partial x_1}(x)
		&
		\frac{\partial^2 Q}{\partial x_2^2}(x)
	\end{pmatrix}.
	\]
	
%	The notation $D^2 Q$ therefore denotes the collection of all second-order
%	partial derivatives of $Q$.
\end{definition}

%\subsection{Laplacian}

\begin{definition}[Laplacian]
	The Laplacian of $Q$ is defined as the trace of the Hessian, that is,	
	\[
	\Delta Q(x)
	=
	\frac{\partial^2 Q}{\partial x_1^2}(x)
	+
	\frac{\partial^2 Q}{\partial x_2^2}(x)=Tr(D^2Q(x)),
	\]
	where $Tr(\cdot)$ denotes the trace.
%	Equivalently,
%	
%	\[
%	\Delta Q = \operatorname{tr}(D^2 Q).
%	\]
\end{definition}

\begin{definition}[Locally bounded function]
	A measurable function $f:\mathbb{R}^2 \to \mathbb{R}$ is said to be a locally bounded function if for every compact set
	$K \subset \mathbb{R}^2$,
		\[
	\|f\|_{L^\infty(K)}
	:=
	\operatorname*{ess\,sup}_{x\in K} |f(x)|
	< \infty .
	\]
	In other words, $f$ is essentially bounded on every compact subset
	of $\mathbb{R}^2$. The space of all locally bounded functions on $\R^2$ is denoted by $L^\infty_{\mathrm{loc}}(\mathbb{R}^2)$.
\end{definition}

\begin{result}\cite[Proposition 2.15]{serfaty2024lectures}\label{re:c11}
A function $Q\in C^{1,1}_{\mathrm{loc}}(\mathbb{R}^2)$ if and only if
\(
D^2 Q \in L^\infty_{\mathrm{loc}}(\mathbb{R}^2),
\)
that is, all second derivatives of $Q$ are locally essentially bounded.
In particular,
\(
\Delta Q \in L^\infty_{\mathrm{loc}}(\mathbb{R}^2).
\)
\end{result}

\subsection{Weighted equilibrium measure}
Let $\mu$ be a probability measure on $\C$. The logarithmic potential of $\mu$ on $\mathbb{C}$ is  defined as 
\[
p_\mu(z)
=
\int \log \frac{1}{|z-t|}\, d\mu(t), \mbox{ for $z \in \C$}.
\]
Let $E \subset \mathbb{C}$ be a closed set, then a non-negative function $w: E \to [0,\infty)$ is called a weight function defined on $\Sigma$. 
\begin{definition}
	A weight function $w:E\to [0,\infty)$  is said to be admissible if it satisfies the following three conditions:
	\begin{enumerate}
		\item $w$ is upper semi-continuous;
		\item $E_{0}=\{z\in E\suchthat w(z)>0\}$ has positive capacity.
		\item if $E$ is unbounded then $|z|w(z)\to 0$ as $|z|\to \infty$, $z\in E$.
	\end{enumerate} 
\end{definition}
\noindent Let $w$ be an admissible weight, and define $Q:\mathbb{R}^2 \to \mathbb{R}$ such that $w(z)=\exp(-Q(z))$. The function $Q$ is known as an {\it externel potential}. The weighted logarithmic energy functional of $\mu$ with externel potential $Q$ is defined as
\[
I_Q(\mu)
=
\iint \log \frac{1}{|z-t|}\, d\mu(z)d\mu(t)
+
2\int Q\, d\mu,
\]
By \cite[Chapter~I]{saff1997logarithmic},
there exists a unique minimizer $\nu$, knonw as {\it the weigted equillibrium measure}, over all probability measure on the complex plane. 
The equilibrium measure is characterized by the Euler--Lagrange conditions:
there exists a constant $C\in\mathbb{R}$ such that
\begin{align}
	p_\nu + Q &\ge C 
	\quad \text{quasi-everywhere in } \mathbb{R}^2,
	\label{EL1} \\
	p_\nu + Q &= C 
	\quad \text{quasi-everywhere on } \operatorname{supp}(\nu).
	\label{EL2}
\end{align}
Moreover, the logarithmic potential satisfies the distributional identity
\begin{equation}
	\Delta p_\mu = -2\pi \mu.
	\label{laplace}
\end{equation}
We refer to  \cite[p. 85]{saff1997logarithmic} for details.

\subsection{Connection with the obstacle problem}
\begin{definition}[Obstacle Problem]
	Let $\Omega \subset \mathbb{R}^n$ be a bounded domain and let 
	$\psi : \Omega \to \mathbb{R}$ be the obstacle. 
	Find a function $u$ such that
	\begin{align*}
		u &\ge \psi && \text{in } \Omega, \\
		-\Delta u &\ge 0 && \text{in } \Omega, \\
		(-\Delta u)(u-\psi) &= 0 && \text{in } \Omega, \\
		u &= g && \text{on } \partial \Omega .
	\end{align*}
\end{definition}

\noindent Define
\(
\Phi := p_\mu + Q - C
\) and $S = \{ p_{\mu_Q} = C - Q \} = \{ \zeta = 0 \}$. The set $S$ is known as the {\it coincidence set}.
Then \eqref{EL1}--\eqref{EL2} imply that 
\[
\Phi \ge 0,
\qquad
\Phi = 0 \quad \text{on } \operatorname{supp}(\mu).
\]
This characterization corresponds to a classical obstacle problem with
obstacle $0$ and forcing term $\Delta Q$.
See \cite{kinderlehrer2000introduction}[Chapter IV] for the general theory of
variational inequalities and obstacle problems.
We state the following result which is needed for our result.

\begin{result}\cite[Proposition 2.15]{serfaty2024lectures} \label{re:seferty} Let $Q$ be a continuous function satisfying the following growth assumption 
	\[
	\lim_{|x|\to\infty}(Q(x)-\log|x|)=\infty
	\]
	Let $\mu_Q$ be the equilibrium measure associated with the potential $Q$.  Then its potential $p_{\mu_Q}$
%	, defined as 
%	\[
%	h_{\mu_V}(z)=\int_{\mathbb{C}}\log{\frac{1}{|z-w|}}d\mu_V(w).
%	\]
	is the unique solution to the obstacle problem with obstacle
	\(
	\psi = C- Q.
	\)
		Moreover,  if $Q \in C^{1,1}(\mathbb R^2)$ then
	\[
	\mu_Q = \frac{1}{c_2} (\Delta Q)\,\mathbf{1}_{S},
	\]
	where $S$ is the coincidence set and $c_2$ is a positive constant. 
\end{result}

  \begin{corollary}\label{cor:measureupper}
  	Let $Q\in C^{1,1}_{loc}(\R^2)$ and $\mu_Q$ be the corresponding weighted equllibrium measure. Then there exists a positive constant $C_2$ such that  
  	\[
  	\mu_Q(A)\le C_2m(A), \mbox{ for all $A\subset \mathbb C$}
  	\]
  	where $m$ is the Lebesgue measure on the complex plane.
  \end{corollary}
  \begin{proof}
  Since $Q\in C^{1,1}_{loc}(\R^2)$, by \Cref{re:c11}, we have $\Delta Q\in L^\infty_{loc}(\R^2)$. In other words, $\|\Delta Q\|_{L^\infty}<\infty$. Therefore by \Cref{re:seferty} we have
  \begin{align*}
  \mu_Q(A)\le \frac{1}{c_2}\|Q\|_\infty m(A\cap S)\le C_2m(A),
  \end{align*}
  where $C_2=\|Q\|_\infty/c_2$. Hence the result.
  \end{proof}
  \begin{corollary}\label{cor:equillibriummeasure}
  	Let $V$ be as in \Cref{ass:V}. Let $\mu_V$ denote the weighted equillibrium measure on $\C$ with external potential $V/2$. Then $\mu_V\in M_1^*(\C)$, as defined in \eqref{eqn:defm1F}. 
  \end{corollary}
  \begin{proof}
Let $\epsilon>0$. Since $V$ satisfies \Cref{ass:V},  	then  \Cref{cor:measureupper} implies that
  	\[
  	\mu_V((\partial U)_\epsilon)\le C_2m((\partial U)_\epsilon),
  	\]
  	for some positive constant $C_2$. On the other hand  $U\subset \C$ satisfies \Cref{ass:U}, then  \Cref{cor:tube_estimate_rectifiable} implies that,   for some positive constant $C_1$, 
  	\[
  	m((\partial U)_\epsilon)\le C_1\epsilon,
  	\]
  	The result follows from the last two equations.
  \end{proof}
 	 
 %%%%%%%%%%%%%%%%%%%%%%%%%%%%%%
\section{Proof of theorem \ref{theorem2}}\label{sec:proofofthm}
In this section we provide the proof of \Cref{theorem2}, assuming \Cref{proposition}.  The following lemmas will be used in the proof of the theorem.
\begin{lemma}\label{lem:openclosed}
	Let $U\subset \mathbb F$ be as in \Cref{theorem2}. For $A\subseteq [0,1]$, define
	\begin{align*}
		\mathcal S_A^*&=\{\mu\in M_1^*(\mathbb F)\suchthat \mu(U)\in A\},
	\end{align*}
		Then, for an open subset $O$ and a closed subset of $F$,
%	\begin{align*}
%		\mathcal S_O^*&=\{\mu\in M_1^*(\mathbb F)\suchthat \mu(U)\in O\},
%		\\\mathcal S_F^*&=\{\mu\in M_1^*(\mathbb F)\suchthat \mu(U)\in F\}.
%	\end{align*}
	  the subsets $\mathcal S_O^*$ and $\mathcal S_F^*$ are open and closed  in $M_1^*(\mathbb F)$ respectively.
\end{lemma}
\begin{lemma}\label{lem:goodrategamma}
	The rate function $\gamma$, as defined in \eqref{eqn:defgamma}, is a good rate function on $[0,1]$.
\end{lemma}

\begin{lemma}\label{lem:fullmeasurefor R}
	Let $M_1^*(\mathbb R)$ and $L_n$ be as defined in \eqref{eqn:defm1F} and \eqref{eqn:defLn} respectively. Then 
	\[
	\P^{\mathbb R}_{n,\beta}(L_n\in M_1^*(\mathbb R))=1, \mbox{ for each $n\in \N$}.
	\]
\end{lemma}

We proceed to prove the theorem assuming the lemmas. The proofs of these lemmas are given in the following two subsections.

\begin{proof} [Proof of Theorem \ref{theorem2}] For ease of writing, we denote $M_n:=n^{-1}\mathcal X_{n,\beta}^{\mathbb F}(U)$.  Observe that,
	for $A\subseteq [0,1]$, we have
	\begin{align}\label{eqn:relation}
			%\l\{M_n\in A\r\}&= \l\{ \frac{1}{n} \sum_{k=1}^{n} \delta_{\lambda_k}(U) \in A\r\}=\l\{L_n(U)\in A\r\}, \mbox{ and }\nonumber \\
		\inf_{x\in A}\gamma(x) &= \inf_{x \in A}\inf_{\mu \in M_1^*(\mathbb{F})}\{I_{\beta}(\mu) \suchthat \mu(U)=x\} = \inf_{\mu\in \mathcal S_A^*} I_{\beta}(\mu).
	\end{align}
	On the other hand, define $\mathcal S_A=\{\mu\in M_1(\mathbb F)\suchthat \mu(U)\in A\}$ for $A\subseteq [0,1]$. Then 
	\begin{align}\label{eqn:sumprobability}
		 \P_{n,\beta}^{\mathbb F}(\{M_n \in A\})&=\P_{n,\beta}^{\mathbb F}(\{L_n(U) \in A\})=\P_{n,\beta}^{\mathbb F}(\{L_n \in \mathcal S_A\})\nonumber
		 \\&=\P_{n,\beta}^{\mathbb F}(\{L_n \in \mathcal S_A\cap M_1^*(\mathbb F)\})+\P_{n,\beta}^{\mathbb F}(\{L_n \in \mathcal S_A\cap M_1^*(\mathbb F)^c\})\nonumber
		 \\&=\P_{n,\beta}^{\mathbb F}(\{L_n \in \mathcal S_A^*\})+\P_{n,\beta}^{\mathbb F}(\{L_n \in \mathcal S_A\cap M_1^*(\mathbb F)^c\}).
	\end{align}
	
	\vspace{.1cm} 
\noindent \underline{\it Proof of (1).}  Let $O$ be an open set in $\mathbb [0,1]$. Then $\mathcal S_O^*$ is open in $M_1^*(\mathbb F)$ by  \Cref{lem:openclosed}. Note that \eqref{eqn:sumprobability} implies that 
\[
\P_{n,\beta}^{\mathbb F}(\{M_n \in O\})\ge \P_{n,\beta}^{\mathbb F}(\{L_n \in \mathcal S_O^*\}).
\]
Then \Cref{proposition},  and \eqref{eqn:relation} imply that 
	\begin{align}\label{eqn:openlower}
		\liminf_{n \to \infty} \frac{1}{n^2}\log \P_{n,\beta}^{\mathbb F}\{M_n \in O\} & \ge \liminf_{n \to \infty} \frac{1}{n^2}\log  \P_{n,\beta}^{\mathbb F}\{L_n \in \mathcal S_{O}^*\}\nonumber
		\\&= -\inf_{\mu\in \mathcal S_O^*}I_{\beta}(\mu) = -\inf_{x\in O} \gamma(x).
	\end{align}
	Hence the result.
	
\vspace{.1cm}
\noindent \underline{\it Proof of (2).}
	If  $F \subseteq [0,1]$ is closed then  $\mathcal S_F^*$ is closed in $M_1^*(\mathbb F)$ by \Cref{lem:openclosed}.

	\noindent \underline{\it Case-I.} (For $\mathbb{F}=\mathbb{R}$). From \Cref{lem:fullmeasurefor R}, it follows that $\P^{\mathbb R}_{n,\beta} (L_n\in \mathcal S_F\cap M_1^*(\R)^c)=0$. 
	Therefore \eqref{eqn:sumprobability}, \Cref{proposition},  and \eqref{eqn:relation}  imply that 
	\begin{align}\label{eqn:closedupper}
		\limsup_{n \to \infty} \frac{1}{n^2} \log \P^{\mathbb R}_{n,\beta} (M_n \in F) & = \limsup_{n \to \infty} \frac{1}{n^2} \log \P_{n,\beta}^{\mathbb R} (L_n \in  \mathcal S_F^*\}) \nonumber
		\\&\leq -\inf_{\mu\in \mathcal S_F^*}I_{\beta}(\mu) = -\inf_{x\in F} \gamma(x).
	\end{align}
	
%	Consider $O$ be an open subset of $[0,1]$. Then
%	\begin{align*}
%			&\liminf_{n \to \infty} \frac{1}{n^2}\log \P_{n,\beta}^{\mathbb F}\{M_n \in O\} \\& =\liminf_{n \to \infty} \frac{1}{n^2}\log  \P_{n,\beta}^{\mathbb F}\{L_n \in \mathcal S_{O}\}\nonumber+\liminf_{n \to \infty} \frac{1}{n^2}\log  \P_{n,\beta}^{\mathbb F}\{L_n \in M_1^*(\mathbb{C})\cap\{\mu\in M_1(\mathbb{C})|\mu(U)\in F\}\}\nonumber
%		\\&\geq \liminf_{n \to \infty} \frac{1}{n^2}\log  \P_{n,\beta}^{\mathbb F}\{L_n \in \mathcal S_{O}\}\nonumber \\&\geq -\inf_{\mu\in \mathcal S_O}I_{\beta}(\mu) = -\inf_{x\in O} \gamma(x).
%	\end{align*} 
\noindent \underline{\it Case-II.} (For $\mathbb{F}=\mathbb{C}$)	 Suppose  $F\subseteq [0,1]$ is a closed set  which satisfies \eqref{eqn:condiononF}. Then \eqref{eqn:sumprobability} implies that,  for $0<\alpha<2$ and for large n,
\[
\P^{\mathbb C}_{n,\beta} (M_n \in F)\le 2e^{n^\alpha}\P^{\mathbb C}_{n,\beta} (L_n \in S_F^*).
\]
Then \Cref{proposition},  and \eqref{eqn:relation}  imply that 
	\begin{align*}
		\limsup_{n \to \infty} \frac{1}{n^2} \log \P^{\mathbb F}_{n,\beta} (M_n \in F) &\leq \limsup_{n \to \infty} \frac{1}{n^2} \log \P^{\mathbb C}_{n,\beta} (L_n \in S_F^*) \\& \leq -\inf_{\mu\in \mathcal S_F^*}I_{\beta}(\mu) = -\inf_{x\in F} \gamma(x). 
	\end{align*}

	Thus, the LDP type bounds hold for the laws of $\{M_n\}$ with speed $n^2$ and the rate function $\gamma$ defined on  $[0,1]$. Furthermore, \Cref{lem:goodrategamma} shows that the rate function $\gamma$ is indeed a good rate function on $[0,1]$. Therefore, we conclude the result.
\end{proof}

\subsection{Proof of \Cref{lem:openclosed}}
Before proving the lemma we recall the well known Portmanteau theorem, which will be used in the proof of \Cref{lem:openclosed}.

\begin{result}[Portmanteau Theorem]\label{re:portmeantau}
	Let  $(\mathbb{F},\mathcal{B}(\mathbb{F}))$
	be equipped with its Borel $\sigma$-algebra induced by the usual topology.
	Let $\{\mu_n\}_{n\ge 1}$ and $\mu$ be probability measures on $\mathbb{F}$.
	Then the following statements are equivalent:
	
	\begin{enumerate}
		\item $\mu_n \Rightarrow \mu$ (weak convergence), i.e.
		\[
		\int f \, d\mu_n \;\longrightarrow\; \int f \, d\mu
		\quad \text{for all } f \in C_b(\mathbb{F}).
		\]
		
		\item For every closed set $F \subset \mathbb{F}$,
		\[
		\limsup_{n\to\infty} \mu_n(F) \le \mu(F).
		\]
		
		\item For every open set $G \subset \mathbb{F}$,
		\[
		\liminf_{n\to\infty} \mu_n(G) \ge \mu(G).
		\]
		
		\item For every Borel set $A \subset \mathbb{F}$ such that
		$\mu(\partial A)=0$,
		\[
		\mu_n(A) \longrightarrow \mu(A).
		\]
	\end{enumerate}
\end{result}
We refer to \cite{billingsley2013convergence} for a proof of \Cref{re:portmeantau}. Now we proceed to prove \Cref{lem:openclosed}.
\begin{proof}[Proof of \Cref{lem:openclosed}]
%	Let $U\subset \mathbb F$ that satisfies \Cref{ass:U}. Then, for $\mathbb F=\mathbb R$, we have the boundary $\partial U$ of $U$ is polar. Since the Lebesgue measure of a polar set is zero, hence we get  $m(\partial U)=0$. 
	Let $M_1^*(\mathbb F)$ be as in \eqref{eqn:defm1F}.
	Define a  map 
	$f:M_1^*(\mathbb F)\to [0,1]$  as 
	\begin{align*}
		f(\mu)= \mu(U).
	\end{align*}
	Observe that if $\mu \in M_1^*(\mathbb F)$ then $\mu (\partial U)=0$.
	The Portmanteau theorem implies that the function $f$ is continuous on $M_1^*(\mathbb F)$ equipped with the weak topology. Then $\mathcal O= f^{-1}(O)$ and $\mathcal F= f^{-1}(F)$ which are open and closed respectively as $f$ is a continuous function. 
	\end{proof}

\subsection{Proof of \Cref{lem:goodrategamma}}
The following lemma will be used in the proof of \Cref{lem:goodrategamma}.
\begin{lemma}\label{lem:goodrateI} 
	Let $I_\beta $ be as defined  in \eqref{eqn:defIbeta}. Then $I_\beta$ is  a good rate function in $M_1^*(\mathbb F)$ equipped with the weak topology.
\end{lemma}

We use the following results to prove \Cref{lem:goodrateI}. 
%\begin{fact}\cite[Lemma 2.4]{serfaty2024lectures}\label{fact:1}
%Let $I_\beta(\mu)$ be as defined in \eqref{eqn:defIbeta}. Suppose $V$ satisfies \Cref{ass:V} then $I_\beta(\mu)$ is a good rate function in $M_1(\mathbb F)$ with weak topology.
%\end{fact}
\begin{result}\cite[Lemma 4.1.5]{dembo2009large}\label{res:dembo}
	Let $\mathcal{Y}$ be a measurable subset of $\mathcal{X}$ such that $\mu_n(\mathcal{Y}) = 1$ for each $n\in \N$. Suppose that $\mathcal{Y}$ is equipped with the topology induced by $\mathcal{X}$. If $\{\mu_n\}$ satisfies the LDP in $\mathcal{X}$ with the rate function I and $\mathcal{D}_I\subset \mathcal{Y}$ then the same LDP holds in $\mathcal{Y}$. Moreover, under the same assumption if I is a good rate function in $\mathcal{X}$, then it is also a good rate function in $\mathcal{Y}$.
\end{result}
\begin{result}(\cite[Theorem 2.6.1]{anderson2010introduction} and \cite[Theorem 1.1]{MR3262506})\label{re:ldp}
	Let $(\lambda_1,\ldots,\lambda_n)$  be distributed according to the law $\P_{n,\beta}^{\mathbb F}$, and $L_n$ be as defined in \eqref{eqn:defLn}.  Then the laws of random measures $\{L_n\}$ satisfies  the LDP in $M_1(\mathbb{F})$ equipped with the weak topology  with  speed $n^2$ and  good rate function $I_\beta$ as defined in \eqref{eqn:defIbeta}.
	That is,
	\begin{enumerate}
		\item for any open set $\mathcal O \subseteq M_1(\mathbb{F})$,
		\[
		\liminf_{n \to \infty} \frac{1}{n^2} \log \P_{n,\beta}^{\mathbb F}(L_n \in \mathcal O) \geq -\inf_{\mu\in \mathcal O} I_\beta(\mu),
		\]
		\item for any closed set $\mathcal F \subseteq M_1(\mathbb{F})$,
		\[
		\limsup_{n \to \infty} \frac{1}{n^2} \log \P_{n,\beta}^{\mathbb F} (L_n \in \mathcal F) \leq -\inf_{\mu\in \mathcal F} I_\beta(\mu).
		\]
	\end{enumerate}
\end{result}
\begin{proof}[Proof of \Cref{lem:goodrateI}]
	\underline{\it Case 1: For $\mathbb{F}=\mathbb{R}$.} From \Cref{re:ldp} we know that $I_{\beta}$ is a good rate function in $M_1(\mathbb R)$ equipped with weak topology.  In other words, for  $\eta\geq 0$, the level set $K_\eta:=\{\mu\in M_1(\R)\suchthat I_{\beta}(\mu)\leq \eta\}$ is compact in $M_1(\mathbb{R})$.  Then it is enough to show that 
	\begin{align}\label{eqn:D_I}
	\mathcal D_I=\{\mu\in M_1(\R)\suchthat I_{\beta}(\mu)<\infty\}\subset M_1^*(\mathbb{R}). 
	\end{align}
	From \eqref{eqn:D_I} we have $K_\eta\subset M_1^*(\mathbb{R})$ for each $\eta\ge 0$. Now from \eqref{res:dembo}, using \eqref{eqn:D_I} and \eqref{lem:fullmeasurefor R} we'll have that $I_{\beta}$ ia a good rate function on $M_1^*(\mathbb{R})$. Thus it remains to show \eqref{eqn:D_I}. 
	
	\vspace{.1cm}
	 \noindent  \underline{\it Proof of \eqref{eqn:D_I}.} 
	On the contrary, suppose the above inequality is not true, which means that there exists a measure $\mu$ such that $I_{\beta}(\mu)<\infty$, but $\mu(\partial U) \neq 
	0$. Then consider the following measure $\Tilde{\mu}$ in the following manner 
	\[
	\Tilde{\mu}(A) = 
	\begin{cases}
		\mu(A), & \text{if } A\subset \partial U, \\
		0, & \text{otherwise }  \\
	\end{cases}
	\]
	Then $\Tilde{\mu}$ is a nonzero measure on $\partial U$ and $\partial U$ is a polar set, which implies from the definition of polar set that $I_{\beta}^V(\Tilde{\mu})=\infty$, which is a contradict the fact that $I_{\beta}^V(\mu)<\infty$. Hence, our assumption was wrong, and the containment is true.  Hence the result.
	
	\vspace{.2cm}
	\noindent \underline{\it Case 2: For $\mathbb F=\C$.} We first show that $M_1^*(\mathbb C)$ is a closed subset of $M_1(\mathbb C)$. Note 
	$$
	M_1^*(\mathbb C)=\bigcap_{0<\epsilon<\delta_0}M_1^{\epsilon}(\mathbb C) \mbox{ where } M_1^\epsilon(\mathbb C)=\{\mu\in M_1(\mathbb C)\suchthat \mu((\partial U)_\epsilon)\le c\epsilon\}.
	$$
	By Portmanteau theorem it follows that $M_1^\epsilon(\mathbb C)$ is a closed set. Indeed, if $\mu_n\in M_1^{\epsilon}(\mathbb C)$ and $\mu_n$ converges to $\mu$ weakly as $n\to \infty$ then 
	\[
	\mu((\partial U)_\epsilon)\le \liminf_{n \to \infty}(\mu(\partial U)_\epsilon)\le c\epsilon,
	\]
	 as $(\partial U)_\epsilon$ is open. Hence $\mu \in M_1^{\epsilon}(\mathbb C)$. Therefore $M_1^*(\mathbb C)$ is a closed set, as arbitrary intersection closed sets is closed. For $\eta>0$, define 
	 \[
	 K^*_\eta=\{\mu \in M_1^*(\mathbb C)\suchthat I_\beta(\mu)\le \eta\}=K_\eta\cap M_1^*(\mathbb C),
	 \]
	 where $K_\eta=\{\mu\in M_1(\mathbb C) \suchthat I_\beta(\mu)\le \eta\}$. Note that the rate function $I_\beta(\cdot)$ is lower semi-continuous and $K_\eta$ is compact in $M_1(\mathbb C)$, as $I_\beta$ is a good rate function in $M_1(\mathbb C)$ by \Cref{re:ldp}. Which implies that the rate function  $I_\beta(\mu)$ is lower continuous on $M_1^*(\mathbb C)$ and $K_\eta^*$ is compact for all $\eta>0$, as intersection of closed and compact set is compact. Hence the result.
	  %Let $\mu_n$ be a sequence of measures in $K_\eta^*$ such that $\mu_n$ converges to $\mu$ weakly as $n\to \infty$.
	\end{proof}

%	\begin{result}\cite[Lemma 2.4]{serfaty2024lectures}
%		Let $\mathcal E:\mathcal P(\mathbb R^d)\to\mathbb R\cup\{+\infty\}$ be the functional
%		\[
%		\mathcal E(\mu)
%		=
%		\frac12
%		\iint_{\mathbb R^d\times\mathbb R^d}
%		g(x-y)\,d\mu(x)d\mu(y)
%		+
%		\int_{\mathbb R^d}V(x)\,d\mu(x),
%		\]
%		defined on the space $\mathcal P(\mathbb R^d)$ of probability measures on $\mathbb R^d$.
%		Assume that
%		
%		\begin{itemize}
%			\item[(A1)] $V$ is lower semicontinuous (l.s.c.) and bounded below,
%			\item[(A2)] (growth assumption)
%			\[
%			\lim_{|x|\to\infty}(V(x)+g(x))=+\infty .
%			\]
%		\end{itemize}
%		
%		Let $\{\mu_n\}_n$ be a sequence in $\mathcal P(\mathbb R^d)$ such that $\{\mathcal E(\mu_n)\}_n$ is bounded. Then, up to extraction of a subsequence, $\mu_n$ converges to some $\mu\in\mathcal P(\mathbb R^d)$ in the weak sense of probabilities, and
%		\begin{equation}
%			\liminf_{n\to\infty}\mathcal E(\mu_n)\ge \mathcal E(\mu).
%		\end{equation}
%		
%		Moreover $\inf \mathcal E > -\infty$. In other words, $\mathcal E$ is lower semicontinuous, bounded below, and its sub-level sets are compact.
%	\end{result}
%	\noindent If we substitute $g(x,y)=-\beta\log|x-y|$ we'll get that $I_\beta$ will be a good rate function on $M_1(\mathbb{F})$ 

\begin{proof}[Proof of \Cref{lem:goodrategamma}]
By \Cref{lem:goodrateI}, we have $I_\beta:{M}_1^*(\mathbb{F}) \to[0,\infty]$ is a good rate function. Thus  $I_\beta$ is lower semicontinuous and $K_\eta^*$ compact for all $\eta>0$.
%	\[
%	K_c := \{\mu\in\mathcal{M}_1^*(\mathbb{F}) : I_\beta(\mu)\le c\}, \qquad c\ge0.
%	\]
	For a measurable set $U\subseteq \mathbb{F}$ that satisfies \Cref{ass:U}, we define 
		\[
	f:{M}_1^*(\mathbb{F})\to[0,1], \qquad f(\mu)=\mu(U).
	\]
	Observe that by the Portmantau theorem it follows that the map $f$ 
	is continuous as $\mu(\partial U)=0$ on $M_1^*(\mathbb F)$.  Recall from \eqref{eqn:defgamma}, for $x\in [0,1]$, we have
	\[
	\gamma(x) := \inf\{ I_\beta(\mu) : f(\mu)=x,\; \mu\in {M}_1^*(\mathbb F)\},
	\]
	with the convention $\inf\emptyset=+\infty$. We show that $\gamma$ is lower semicontinuous and has compact sublevel sets.

	\noindent {\it Lower semicontinuity.}
	Let $x_n\to x$ in $[0,1]$. 
	We need to  show that  
	$$
	\liminf_{n\to\infty}\gamma(x_n)\ge \gamma(x).
	$$
	Note that if $\liminf_{n}\gamma(x_n)=+\infty$, there is nothing to prove.  Thus we assume that $\alpha:=\liminf_{n}\gamma(x_n)<\infty$.  Then  there exists a subsequence $\{n_k\}$ such that $\gamma(x_{n_k})$ converges to $\alpha$.
	By the definition of $\gamma$ it follows that for each $n$ there exists $\mu_n\in {M}^*_1(\mathbb F)$ such that
	\[
	f(\mu_n)=x_n, \qquad I_\beta(\mu_n)\le \gamma(x_n)+\tfrac1n.
	\]
	In particular we have  $\sup_{n_k} I_\beta(\mu_{n_k})<\infty$. Which implies that if  $\eta>\alpha$ then $I_\beta(\mu_{n_k})\le \eta$ for all large $n_k$. Therefore $\mu_{n_k}\in K_\eta^* $ for all sufficiently large $n_k$. As  $K_\eta^*$ is compact, there exists a subsequence $\{n_{k_j}\}$ and $\mu\in K_\eta^*$ such that 
	\[
	\mu_{n_{k_j}} \to \mu, \mbox{ weakly as $j\to \infty$}.
	\]
		By the lower semicontinuity of $I_\beta$, we have
		\begin{align}\label{eqn:liminf}
		I_\beta(\mu)\le \liminf_{n\to\infty} I(\mu_{n_{k_j}}) = \liminf_{n\to\infty} \gamma(x_{n_{k_j}})=\alpha.
		\end{align}
	On the other hand, by the continuity of $f$ on ${M}_1^*(\mathbb F)$, we have 
	\[
	f(\mu)=\lim_{n\to\infty} f(\mu_n)=\lim_{n\to\infty} x_n = x.
	\]
	Thus $\mu\in {M}_1^*(\mathbb F)$ and $f(\mu)=x$,  consequently, by \eqref{eqn:liminf} we get
	\[
	\gamma(x) \le I(\mu) \le \alpha=\liminf_{n\to\infty}\gamma(x_n).
	\]
	Therefore, $\gamma$ is lower semicontinuous.
	
	\vspace{.2cm}
	\noindent{\it Compactness of sublevel sets.}
	For each $\eta\ge0$,  
	we show that $A=B$, where
	\[
	A:=\{x:\gamma(x)\le \eta\} 
	\mbox{ and } B:=f\big( K_\eta\big).
	\]
	Suppose $x\in B$. Then there exists $\mu_0\in K_\eta$ such that $f(\mu_0)=x$. Again, $\mu_0\in K_\eta$ implies that $I_\beta(\mu_0)\le \eta$. Thus $\gamma(x)\le \eta$, that is, $x\in A$. Consequently, we have 
	\begin{align}\label{eqn:asubsetb}
	B\subseteq A.
	\end{align}
	On the other hand, suppose $x\in A$. Then we have following two cases.

\vspace{.1cm}
	\noindent{\it Case-I}: Note that if $\gamma(x)<\eta$ then there exists $\mu'\in  M_1^*(\mathbb F) $ such that $f(\mu')=x$ and $I_\beta(\mu')\le \eta$. Then $\mu'\in K_\eta$ and $x\in B$. Thus $A\subseteq B$. 
	
	\vspace{.1cm}
	\noindent{\it Case-II}:
	Suppose, $\gamma(x)=\eta$. Then there exists a sequence  $\{\mu_n\}$ in $M_1^*(\mathbb F)$ such that 
	\[
	\mbox{$\lim_{n\to \infty}I_\beta(\mu_n)= \eta$, and $f(\mu_n)=x$ for all $n$.}
	\]
	 Which implies that $I_\beta(\mu_n)\le \eta+1$ for all large $n$. In other words, $\mu_n\in K_{\eta+1}$ for large $n$. Then, as $K_{\eta+1}$ is compact, there exists a subsequence $\{\mu_{n_k}\}$ and a measure $\mu\in M_1^*(\mathbb F)$ such that $\mu_{n_k}$ converges to $\mu$ weakly as $k\to \infty$. Since $I_\beta$ is lower semi-continuous, we have 
	\begin{align*}
	I_\beta(\mu)\le \liminf_{k\to \infty}I_\beta(\mu_{n_k})=\eta.
	\end{align*}
	Which implies that $\mu\in K_\eta$. By continuity of $f$, we get 
	\[
	f(\mu)=\lim_{n\to \infty} f(\mu_n)=x.
	\]
Which implies that $x\in B$. Thus $A\subseteq B$. Hence by \eqref{eqn:asubsetb} we  get $A=B$, that is,
\[
\{x\suchthat \gamma(x)\le \eta\}=f(K_\eta).
\]
Since the set $K_\eta$ is compact, and $f$ is continuous on it; hence $f(K_\eta)$ is compact in $[0,1]$.
	Therefore each sublevel set $\{\gamma\le \eta\}$ is compact. Thus
	the function $\gamma:[0,1]\to[0,\infty]$ is lower semicontinuous and has compact level sets. In other words, the rate function $\gamma$ is a good rate function.
\end{proof}
\subsection{Proof of \Cref{lem:fullmeasurefor R}} The result follows from the fact that $\P^{\mathbb R}_{n,\beta}$ is absolutely continuous with respect to the Lebesgue measure on $\R^n$. 
\begin{proof}[Proof of \Cref{lem:fullmeasurefor R}]
	It  is enough to check $\P^{\mathbb R}_{n,\beta}(L_n(\partial U)>0)=0$. Since $\partial U$ is polar and the Lebesgue measure of a polar set is zero, therefore $m(\partial U)=0$. Then 
	\begin{align*}
		\P^{\mathbb R}_{n,\beta}(L_n(\partial U)>0) & = \P^{\mathbb R}_{n,\beta}(\frac{1}{n} \sum_{k=1}^{n} \delta_{\lambda_k}(\partial U)>0) \\&= \P^{\mathbb R}_{n,\beta}(\omega| \exists k\in \{1,2,...,n\} ,\lambda_k(\omega)\in \partial U) \\& \leq \sum_{k=1}^{n}\P^{\mathbb R}_{n,\beta}(\lambda_k\in \partial U) = 0.
	\end{align*}
	The summation in the last inequality is $0$ because each of the terms is zero, as $m(\partial U)=0$ and the measure $\P^{\mathbb R}_{n,\beta}$ is absolutely continuous with respect to the Lebesgue measure.
	Hence the result.
\end{proof}

\section{Proof of  \Cref{proposition} for $\mathbb F=\mathbb R$} \label{sec:proofofprop}
In this section we present a proof of  \Cref{proposition} for $\mathbb F=\mathbb R$.  %The following results will be used in the proof of  the proposition.

%\begin{lemma}\label{lem2}
%	Let $M_1^*(R)$ denotes the space of all probability measures in $\mathbb{R}$ which takes value 0 at $\partial U$, that is,
%	$$
%	M_1^*(\mathbb{R}) = \{ \mu \in M_1(\mathbb{R})| \mu(\partial U)=0\}.
%	$$ 
%	Then the sequence of empirical random measures $\{L_n\}_{n \in \mathbb{N}}$ satisfies the Large Deviation Principle in $M_1^*(\mathbb{R})$ with the rate function $I_{\beta}^V$ and with the rate V. 
%\end{lemma}

\begin{lemma}\label{lem:measurable}
	Let $U\subset \R$ which statisfies \Cref{ass:U}, and $M_1^*(\R)$ be as defined in \eqref{eqn:defm1F}.
	%Recall $M_1^*(\mathbb R)=\{\mu\in M_1(\mathbb R)\suchthat \mu(\partial U)=0)\}$. 
	Then $M_1^*(\mathbb{R})$
	is a Borel (hence measurable) subset of $M_1(\mathbb{R})$ equipped with weak topology.
\end{lemma}

\begin{proof}[Proof of \Cref{proposition} for $\mathbb F=\R$]
Note that \Cref{re:ldp} gives the LDP for $\{L_n\}$ in $M_1(\mathbb R)$, whereas we need to show it holds in $M_1^*(\mathbb R)$.
Thus we utilize Result \ref{res:dembo} to prove \Cref{proposition} for $\mathbb F=\R$. In this case $\mathcal X=M_1(\mathbb R)$ and $\mathcal Y=M_1^*(\mathbb R)$. \Cref{lem:measurable} implies that $M_1^*(\R)$ is measurable. Therefore it is enough to show that the two conditions in Result \ref{res:dembo} are satisfied. That is,
\begin{enumerate}
	\item[(i)] $\P^{\mathbb R}_{n,\beta}(L_n\in M_1^*(\mathbb R))=1$, \mbox{ for each $n\in \N$}.
	\item [(ii)] $\mathcal D_I=\{\mu\in M_1(\R)\suchthat I_{\beta}^V(\mu)<\infty\}\subset M_1^*(\mathbb{R})$.
\end{enumerate}
Note $(i)$ is same as \Cref{lem:fullmeasurefor R} and $(ii)$ is same is same as \eqref{eqn:D_I}. Hence the result.
\end{proof}

\begin{proof}[Proof of \Cref{lem:measurable}]
By the Portmanteau theorem, if \(\mu_n \to \mu\) weakly then, for every closed set \(F \subset \mathbb{R}\) and every open set \(O \subset \mathbb{R}\),
\[
\limsup_{n \to \infty} \mu_n(F) \le \mu(F)
\;\; \mbox{ and }\;\;
\liminf_{n \to \infty} \mu_n(G) \ge \mu(G).
\]
Therefore, the map 
\(
\mu \mapsto \mu(F)
\)
is \emph{upper semicontinuous} for every closed set \(F\), and 
\(
\mu \mapsto \mu(O)
\)
is \emph{lower semicontinuous} for every open set \(O\). In particular, both maps are Borel measurable.
Fix an open set \(U \subset \mathbb{R}\). As
\(
\partial U = \overline{U} \setminus U^\circ,
\)
therefore for every probability measure \(\mu\) we have 
\[
\mu(\partial U) = \mu(\overline{U}) - \mu(U^\circ).
\]
For each integer \(k \ge 1\), define
\[
A_{U,k} := \left\{ \mu \in M_1(\mathbb{R}) : \mu(\overline{U}) - \mu(U^\circ) < \frac{1}{k} \right\}.
\]
Because \(\mu \mapsto \mu(\overline{U})\) is upper semicontinuous and 
\(\mu \mapsto \mu(U^\circ)\) is lower semicontinuous, their difference 
\(\mu \mapsto \mu(\overline{U}) - \mu(U^\circ)\) is Borel measurable.
Hence each \(A_{U,k}\) is a Borel subset of \(M_1(\mathbb{R})\).
Consequently,
\[
M_1^*(\R)=\{\mu\in M_1(\R) : \mu(\partial U)=0\} = \bigcap_{k \ge 1} A_{U,k}
\]
is Borel. Hence the result.
\end{proof}

\section{Proof of  \Cref{proposition} for $\mathbb F=\mathbb C$}\label{sec:proof}
 This section focuses on proving \Cref{proposition} for $\mathbb F=\mathbb C$. It is noteworthy to mention that the approach used to prove the proposition for  $\mathbb F=\mathbb R$ can not be used when $\mathbb F=\mathbb C$, as
 \(
\P_{n,\beta}^{\C}(L_n\in M_1^*(\C))\neq 1. 
 \)
 In fact, it can be shown that $\P_{n,\beta}^{\C}(L_n\in M_1^*(\C))\to 1$ as $n\to \infty$. Therefore \Cref{res:dembo} does not hold for $\mathbb F=\mathbb C$. Hence we provide a direct proof, inspired by the proof techniques from \cite{ben1998large} and \cite{MR3262506}. However, the construction of points $\{z_1^*,\ldots,z_n^*\}$ as in \Cref{lem:construction of points} and consequent calculations in this paper are  different from those  in \cite{ben1998large} and \cite{MR3262506}. First we prove the proposition using the following lemmas and facts. The proofs of the lemmas are given at the end of this section.

\begin{lemma}\label{lem:construction of points}
	Let $\nu$ be a probability masure on a square $D$ in the complex plane with density function $g_\nu$. Suppose there exists a constant $C$ such that $1/C\le g_\nu\le C$ on $D$. For each $n\in \N$, then there exists  $z_1^*,\ldots,z_n^*$ in $D$ such that
	\begin{enumerate}
		\item $\frac{1}{n}\sum_{k=1}^n\delta_{z_k^*}$ converges weakly to $\nu$ as $n\to \infty$.
		
		\vspace{.1cm}
		\item  $\min\{|z_i^*-z_j^*|\suchthat 1\le i,j\le n\}\ge \frac{1}{2\sqrt{Cn}}$ for large $n$.
	\end{enumerate}
\end{lemma}

\begin{lemma}\label{lem:energy}
	Let $\nu$ and $z_1^*,\ldots,z_n^*$ be as defined in \Cref{lem:construction of points}. Then
	\[
	\iint \log |z-w|d{\nu}(z)d{\nu}(w)\leq \liminf_{n\to \infty}\l\{\frac{1}{n^2}\sum_{i\neq j=1}^n\log|z_{i}^*-z_{j}^*|\r\}.
	\]
\end{lemma}

\begin{lemma}\label{lem:openset}
	Let $\nu$ be a measure as in \Cref{lem:construction of points}. Then, for any $\delta>0$,
	\[
	\liminf_{n\to \infty}\frac{1}{n^2} \log\P^{\mathbb C}_{n,\beta}(L_n\in B(\nu, \delta))\geq -I_\beta(\nu).
	\]
\end{lemma}

\begin{fact}\label{fact:2}\cite{ben1998large}
	Assume that $\mu \in M_1(\mathbb{C})$ does not possess atoms, and that 
	$I(\mu) < \infty$. Let $\mu_\varepsilon = \mu * \gamma_\varepsilon$, 
	where $\gamma_\varepsilon$ is the standard centered Gaussian law on 
	$\mathbb{C}^2$ of $\varepsilon I$ covariance 
	($*$ denotes the convolution). Then $\mu_\varepsilon \rightarrow \mu$ and 
	\[
	I_\beta(\mu_\varepsilon) \rightarrow I_\beta(\mu), \mbox{ as $\varepsilon \to 0$}.
	\tag{2.3}
	\]
\end{fact}

\begin{fact}\label{fact:3}\cite{MR3262506}
	Let $Z_{n,\beta}^{\C} $ be the partition function as defined in \eqref{eqn:Z}. Then
	\begin{align}\label{73} 
		\limsup_{n \to \infty} \frac{1}{n^2} \log Z_{n,\beta}^{\C} = - \inf_{\mu \in M_1^*(\mathbb{C})} I_\beta(\mu).
	\end{align}
\end{fact}

Now we are ready  to present a prove the proposition for $\mathbb F=\mathbb C$.
\begin{proof}[Proof of \Cref{proposition} for $\mathbb F=\C$]
	This proof follows the same approach as in \cite{ben1998large} and \cite{MR3262506}. However, for the sake of completeness, we provide a proof here.
	
	\vspace{.1cm}
	\noindent \underline{\it Lower bound:}
	First we  show the lower bound of the large deviation, that is, the following holds for any open subset $\mathcal{O}\subset \mathcal M_1^*(\mathbb{C})$
	\[
	\liminf_{n\to \infty}\frac{1}{n^2} \log\P_{n,\beta}^{\C}(L_n\in \mathcal{O})\geq -\inf_{\mu\in \mathcal{O}}I_\beta(\mu).
	\]
	To prove the last inequality, enough to show that for $\delta>0$,
	\begin{align}\label{eqn:singleton}
		\liminf_{n\to \infty}\frac{1}{n^2} \log\P_{n,\beta}^{\C}(L_n\in \mathcal{B(\mu,\delta)})\geq -I_\beta(\mu), \mbox{ for all } \mu\in \mathcal O.
	\end{align}
	Indeed, suppose $\mu\in \mathcal O$ then there exists $\delta>0$ such that $\mathcal B(\mu,\delta)\subset \mathcal O$.
	Thus 
	\begin{align*}
		\liminf_{n\to \infty}\frac{1}{n^2} \log\P_{n,\beta}^{\C}(L_n\in \mathcal{O})\geq 	\liminf_{n\to \infty}\frac{1}{n^2} \log\P_{n,\beta}^{\C}(L_n\in \mathcal{B(\mu,\delta)})\geq -I_\beta(\mu),
	\end{align*}		
	for all $\mu\in \mathcal O$. Hence the result.		
	
	 \noindent It remains to show \eqref{eqn:singleton}. Observe that \eqref{eqn:singleton}  holds trivially if $I_\beta(\mu)=\infty$. Thus it is enough to consider the measure $\mu $ which is non atomic.
	\Cref{lem:openset} implies that \eqref{eqn:singleton} holds if $\mu $ is compactly supported with positive density on it. Suppose $\mu$ is a probability measure such that it has  positive density everywhere  and $I_\beta(\mu)<\infty$.  Then we can consider a sequence $\{\mu_n\}$ of measures which are supported on a square with positive density everywhere that are converging to  $\mu$ as $n \to \infty $. That can be done if we consider a sequence of $D_n$ of squares in the complex plane and which are increasing to the $\mathbb{C}$, i.e. $D_n\subset D_{n+1}$ for all $n\in \mathbb{N}$ and $\cup_{1}^{\infty}D_n=\mathbb{C}$. With this construction in hand, we  create the following sequence of probability measures 
	\[
	\mu_n=\frac{\mu_{{|}_{D_n}}}{\mu(D_n)}, \mbox{ for each $n\in \mathbb{N}$.}
	\] 
	Since $\mu_n$ converges to $\mu$, there exists $N\in \N$ (depends on $\delta$) such that 
	\[
	B(\mu_n,\delta/2)\subseteq B(\mu,\delta), \mbox{ for all } n\ge N.
	\]
	Therefore \eqref{eqn:singleton} implies that 
	\begin{align*}
		\liminf_{n \to \infty}\frac{1}{n^2}\log P_{n,\beta}^{\mathbb C}(L_n\in B(\mu,\delta))&\ge  \liminf_{n \to \infty}\frac{1}{n^2}\log P_{n,\beta}^{\mathbb C}(L_n\in B(\mu_n,\delta))
		\\&\ge -\lim_{n\to \infty} I_{\beta}(\mu_n)= -I_\beta(\mu).
	\end{align*}
	The last equality follows from the monotone convergence theorem. Thus \eqref{eqn:singleton} holds when measure has a positive continuous density in the complex plane.
	
	\noindent Finally, consider the measure $\mu$ such that $I_\beta(\mu)<\infty$. Which implies that $\mu$ does not possess atoms. Then Fact
	\ref{fact:2} implies that there exists a sequence of measures $\mu_n$ with positive density in the complex plane such $\mu_n$ converges to $\mu $ weakly and $I_\beta(\mu_n)\to I_\beta(\mu)$ as $n\to \infty$. Then by the same argument as above we conclude that this $\mu$ satisfies \eqref{eqn:singleton}. Thus \eqref{eqn:singleton} holds for every $\mu \in M_1^*(\mathbb C)$. Hence the lower bound.
	
	\vspace{.1cm}
	\noindent \underline{\it Upper bound:}  Let $\overline{P}^{\C}_{n, \beta} = Z_{n,\beta}^{\C}\P_{n, \beta}^{\C}$. First we show that the following holds
	\begin{align}\label{72}
		\lim_{\epsilon \to 0} \limsup_{n \to \infty} \frac{1}{n^2} \log \overline{P}^{\C}_{n, \beta} \left( d(L_n, \mu) \leq \epsilon \right) \leq - I_\beta(\mu), \mbox{ for all $\mu \in M_1(\mathbb{C})$}.
	\end{align}
	\noindent Denote $f(z,w):=\frac{1}{2}V(z)+\frac{1}{2}V(w)-\frac{\beta}{2}\log |z-w|$, for $z,w\in \C$. For any $M \geq 0$, set $f_M(z,w) = f(z,w) \wedge M$. Then, we have the following bound
	\begin{align}\label{71}
		\overline{P}^{\C}_{n, \beta} \left( d(L_n, \mu) \leq \epsilon \right) \leq \idotsint\limits_{d(L_n, \mu) \leq \epsilon} e^{-n^2 \iint\limits_{z \neq w} f_M(z,w) dL_n(z) dL_n(w)} \prod_{i=1}^n e^{-V(\lambda_i)} d\lambda_i.
	\end{align}
 Note that the $\{\lambda_1,\ldots, \lambda_n\}$ are distinct almost surely (with respect to the Lebesgue measure). Then $L_n \otimes L_n(x = y) = n^{-1}$, $\overline{P}^{n}_{V, \beta}$ almost surely. Thus
	\begin{align}\label{eqn:upper}
		\iint\limits_{z,w\in \C} f_M(z,w) dL_n(z) dL_n(w) = \iint\limits_{z \neq w} f_M(z,w) dL_n(z) dL_n(w) + M n^{-1}.
	\end{align}
	Therefore, using \eqref{eqn:upper} in \eqref{71}, we get
	\begin{align*}
		\overline{P}^{\C}_{n, \beta} \left( d(L_n,\mu)
		\leq \epsilon \right) 
		&\leq e^{Mn} \idotsint\limits_{d(L_n, \mu) \leq \epsilon} e^{-n^2 \iint\limits_{z,w\in \C} f_M(z,w) dL_n(z) dL_n(w)} \prod_{i=1}^n e^{-V(\lambda_i)} d\lambda_i
		\\&=e^{Mn} \idotsint\limits_{d(L_n, \mu) \leq \epsilon}e^{-n^2I^M_\beta(L_n)}\prod_{i=1}^n  e^{-V(\lambda_i)} d\lambda_i,
	\end{align*}
	where $I_\beta^M(\nu)=\iint f_M(z,w)d\nu(z)d\nu(w)$, for all $\nu\in M_1(\mathbb{C})$ is continuous on $M_1({\mathbb{C}})$. Now choose $\epsilon=\frac{1}{m_k}>0$ sufficiently small enough so that $I_\beta^M(L_n)+\frac{1}{k}\geq I_\beta^M(\mu)$ for $d(L_n,\mu)
	\leq\frac{1}{m_k}$ and for $\frac{1}{k}$, for all $k\in \mathbb{N}$. Applying this inequality we'll have
	 \begin{align*}
	 		\overline{P}^{\C}_{n, \beta} \left( d(L_n,\mu)
	 	\leq \frac{1}{m_k} \right) 
	 	& \le e^{-n^2I^M_\beta(\mu)+\frac{n^2}{k}}e^{Mn}\idotsint\limits_{d(L_n, \mu) \leq  \frac{1}{m_k} }\prod_{i=1}^n e^{-V(\lambda_i)} d\lambda_i
	 	\\&\leq Ce^{-n^2I^M_\beta(\mu)+\frac{n^2}{k}}e^{Mn}
	 \end{align*}
	where $C=\idotsint\limits_{d(L_n, \mu) \leq \frac{1}{m_k}}\prod_{i=1}^n  e^{-V(\lambda_i)} d\lambda_i$. Then we have
	\begin{align}
		\lim_{k\to\infty} \limsup_{n \to \infty} \frac{1}{n^2} \log \overline{P}^{\C}_{n, \beta} \left( d(L_n, \mu) \leq \frac{1}{m_k} \right) \leq -I^M_\beta(\mu),
	\end{align}
	where $I_\beta^M(\mu)= \iint f_M(x,y)d\mu(x)d\mu(y)$. Therefore by Fact \ref{fact:3} we have 
	\begin{align*}
		\lim_{k\to\infty} \limsup_{n \to \infty} \frac{1}{n^2} \log \P^{\C}_{n, \beta} \left( d(L_n, \mu) \leq \frac{1}{m_k} \right) \leq -(I^M_\beta(\mu)-c_\beta).
	\end{align*}
	By the monotone convergence theorem taking $M\to \infty$ we get
	\[
	\lim_{k\to\infty} \limsup_{n \to \infty} \frac{1}{n^2} \log \P^{\C}_{n, \beta} \left( d(L_n, \mu) \leq \frac{1}{m_k} \right) \leq -I_\beta(\mu).
	\]
	Thus \eqref{eqn:upper} holds for all $\mu \in M_1(\C)$. In particular \eqref{eqn:upper} holds for all $\mu \in M_1^*(\C)$. Hence the upper bound and the result.
	Now consider a compact set $K\subset M_1^*(\mathbb{C})$, then for any $\epsilon>0$ and for any $\mu\in K$ we have $
	\delta_{\mu}>0$ such that 
	\[
	\limsup_{n \to \infty} \frac{1}{n^2} \log \P^{\C}_{n, \beta} \left( d(L_n, \mu) \leq \delta_{\mu} \right) \leq -I_\beta(\mu)-\epsilon.
	\]
	Now $\{B(\mu,\delta_{\mu})_{\mu\in K}\}$ forms a cover for K and compactness of K ensures that there will be finite subcover say  $\{B(\mu_i,\delta_{\mu_i})_{1\leq i \leq m}\}$. Then we have that  
	\begin{align*}
	\P^{\C}_{n, \beta} ( L_n\in K)\leq \sum_{1}^{m}\P^{\C}_{n, \beta} ( L_n \in B(\mu_i, \delta_{\mu_i}))\leq m\times \max_{1}^{m}\P^{\C}_{n, \beta} ( L_n \in B(\mu_i, \delta_{\mu_i})).
	\end{align*}
	Suppose maximum attains at $i_*$. Then taking normalized logarithmic limit we get
	\begin{align*}
			\limsup_{n \to \infty} \frac{1}{n^2} \log\P^{\C}_{n, \beta} ( L_n\in K)&\leq 	\limsup_{n \to \infty} \frac{1}{n^2} \log\P^{\C}_{n, \beta}( L_n \in B(\mu_{i_*}, \delta_{\mu_{i_*}}))\\&\leq -I_{\beta}(\mu_{i_*})-\epsilon \leq -\inf_{\mu\in K}I_{\beta}(\mu)-\epsilon.
	\end{align*}
	as $\epsilon>0$ is arbitrary, tending $\epsilon\to 0$ we get the upper bound for the compact sets. Which gives the weak LDP. 
	%Then  we conclude the full LDP holds on $M_1^*(\C)$ using the exponential tightness of $\{L_n\}$ on  $M_1^*(\C)$.
	 It remains to show that the upperbounds hold for any closed sets of $M_1^*(\C)$.
	
	\vspace{.1cm}
	\noindent \underline{\it Proof of large deviation upper bound for general closed sets}.
We know that the laws of $\{L_n\}$ are exponentially tight in $M_1(\mathbb{C})$ (see \cite[Theorem 2.6.1]{anderson2010introduction} and \cite[Lemma 2.3]{MR3262506}). It means that for an arbitrary large $M>0$ there exists a compact set $K_M\subset M_1({\mathbb{C}})$ so that 
$$
\limsup_{n \to \infty}\frac{1}{n^2}\log \P_{n,\beta}^{\mathbb{C}}(L_n\in K_{M}^c)\leq -M.
$$
\\ %Consider $K^{'}=K_1\cap M_1^*(\mathbb{C})$. Note that $K'$ is compact in $M_1^*(\mathbb{C})$, as . Now the following holds 
%\begin{align*}
	%&\limsup_{n \to \infty}\frac{1}{n^2}\log P_{n,\beta}^{\mathbb{C}}(L_n\in {K'}^c)
%	\\&\leq \limsup_{n \to \infty}\frac{1}{n^2}\log P_{n,\beta}^{\mathbb{C}}(L_n\in \overline{M_1^*(C))}+\limsup_{n \to \infty}\frac{1}{n^2}\log P_{n,\beta}^{\mathbb{C}}(L_n\in K_1)
%	\\& \leq -L.
%\end{align*}
%Now using the above inequality and the fact that finitely many probability measures are exponentially tight, we have the desired result.
\noindent Consider a closed subset $\mathcal F$ of $M_1^*(\mathbb{C})$. Now the closedness of $M_1^*(\mathbb{C})$ in $M_1(\mathbb{C})$ will imply that $\mathcal F$ is also closed in $M_1(\mathbb{C})$. Then we have
\begin{align*}
	&\limsup_{n\to\infty}\frac{1}{n^2} \log \P_{n,\beta}^{\mathbb{C}}(L_n\in \mathcal F)\\&=\limsup_{n\to\infty}\frac{1}{n^2} \log \P_{n,\beta}^{\mathbb{C}}(L_n\in \{\mathcal F\cap K_M\}\cup \{\mathcal F\cap K_{M}^c\})\\&\leq \max\{\limsup_{n\to\infty}\frac{1}{n^2} \log \P_{n,\beta}^{\mathbb{C}}(L_n\in \{\mathcal F\cap K_M\}),\limsup_{n\to\infty}\frac{1}{n^2} \log \P_{n,\beta}^{\mathbb{C}}(L_n\in \{\mathcal F\cap K_{M}^c\})\}\\&\leq  \max\{\limsup_{n\to\infty}\frac{1}{n^2} \log \P_{n,\beta}^{\mathbb{C}}(L_n\in \{\mathcal F\cap K_M\}),\limsup_{n\to\infty}\frac{1}{n^2} \log \P_{n,\beta}^{\mathbb{C}}(L_n\in K_{M}^c)\}\\& \leq  \max\{\limsup_{n\to\infty}\frac{1}{n^2} \log \P_{n,\beta}^{\mathbb{C}}(L_n\in \{\mathcal F\cap K_M\}),-M\}\\&\leq \max \{-\inf_{\mathcal F\cap K_M}I_{\beta}(\mu),-M\}\leq \max \{-\inf_{\mathcal F}I_{\beta}(\mu),-M\}.
\end{align*}
\noindent Now taking M tending to infinity gives us the large deviation upper bound for the laws of  $L_n$ in $M_1^*(\mathbb{C})$. This completes the proof.
\end{proof}

\subsection{Proofs of the auxiliary lemmas}
In this section we provide the proofs of  \Cref{lem:construction of points}, \Cref{lem:energy} and \Cref{lem:openset}.
\begin{proof}[Proof of \Cref{lem:construction of points}] The proof is divided in three parts. First, we construct the points that satisfy both properties of the lemma. Second, we show that the empirical measure of these points converges to the measure $\nu$. Finally, we show that the separation of these points.

	\noindent{\it Construction of points.}	{\it Case I: }
	First we construct $n$ points when $n=m^2$ for some $m\in \mathbb{N}$.  Without loss of generality we assume that the four vertices of D are $(0,0), (L,0), (L,L), (0,L)$ as shown in Figure 1, where $L$ denotes the side length of $D$.  First we divide $D$ vertically such that the measure of each rectangle is $1/m$. See Figure 1. More precisely, choose $0=x_0, x_1,\ldots, x_m=L$ such that 
	\[
	\nu([x_{i-1},x_i]\times [0,L])=\frac{1}{m}=\frac{1}{\sqrt n}, \mbox{ for all } i=1,\ldots, m-1.
	\]
	This can be done using the continuity of density function for $\nu$. Then each rectangle $\bar R_i=[x_{i-1},x_i]\times [0,L]$ is divided horizontally such a way that the measure of each box is $1/m^2$. That is, for each $1\le i\le m$, we choose $0=y_{i0},y_{i1},\ldots,y_{im}=L$ such that, for each  $1\le j\le m$, 
	\begin{align}\label{def:Rij}
		\nu(R_{i,j})=\frac{1}{m^2}=\frac{1}{n}, \mbox{ where }	R_{ij}=[x_{i-1},x_i]\times [y_{i(j-1)},y_{ij}].
	\end{align}
	In this case, we consider $\{z_{ij}^*\suchthat 1\le i,j\le m\}$, where $z_{ij}^*$ is the intersection point of two diagonals of $R_{ij}$. We arrange these points in dictionary order and rename them as $\{z_1^*,\ldots,z_n^*\}$.
	
	%%%%%%%%%%%\begin{figure}[h]
	%%%%%%%%%%%%%\centering
	%%%%%%%%%%%%%	\includegraphics[width=0.5\linewidth]{Notes_251121_015420_page-0001.jpg}
	%%%%%%%%%%%%%%%%%%%	\caption{{}}
	%%%%%%%%%%%%%%	\end{figure}

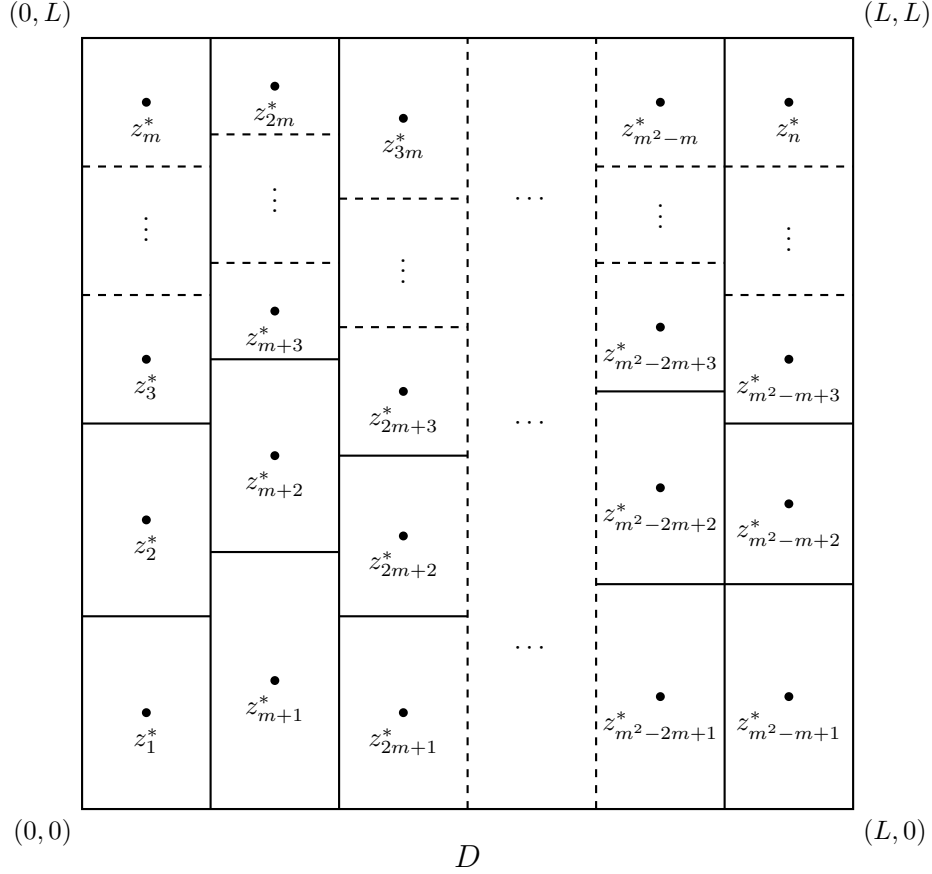
\begin{figure}[htbp]
	\centering
	\begin{tikzpicture}[scale=0.85, thick]
		% Outer boundary of square D
		\draw (0,0) rectangle (12,12);
		
		% Vertical partition lines
		\draw (2,0) -- (2,12);
		\draw (4,0) -- (4,12);
		\draw[dashed] (6,0) -- (6,12); % Start of the gap
		\draw[dashed] (8,0) -- (8,12); % End of the gap
		\draw (10,0) -- (10,12);
		
		% --- Column 1 ---
		\draw (0,3) -- (2,3);
		\draw (0,6) -- (2,6);
		\draw[dashed] (0,8) -- (2,8);
		\draw[dashed] (0,10) -- (2,10);
		\filldraw (1, 1.5) circle (1.5pt) node[below=2pt] {$z^*_1$};
		\filldraw (1, 4.5) circle (1.5pt) node[below=2pt] {$z^*_2$};
		\filldraw (1, 7) circle (1.5pt) node[below=2pt] {$z^*_3$};
		\node at (1, 9.15) {$\vdots$};
		\filldraw (1, 11) circle (1.5pt) node[below=2pt] {$z^*_m$};
		
		% --- Column 2 ---
		\draw (2,4) -- (4,4);
		\draw (2,7) -- (4,7);
		\draw[dashed] (2,8.5) -- (4,8.5);
		\draw[dashed] (2,10.5) -- (4,10.5);
		\filldraw (3, 2) circle (1.5pt) node[below=2pt] {$z^*_{m+1}$};
		\filldraw (3, 5.5) circle (1.5pt) node[below=2pt] {$z^*_{m+2}$};
		\filldraw (3, 7.75) circle (1.5pt) node[below=2pt] {$z^*_{m+3}$};
		\node at (3, 9.6) {$\vdots$};
		\filldraw (3, 11.25) circle (1.5pt) node[below=2pt] {$z^*_{2m}$};
		
		% --- Column 3 (Fixed: Added parallel dashed line & dot) ---
		\draw (4,3) -- (6,3);
		\draw (4,5.5) -- (6,5.5);
		\draw[dashed] (4,7.5) -- (6,7.5);
		\draw[dashed] (4,9.5) -- (6,9.5);
		\filldraw (5, 1.5) circle (1.5pt) node[below=2pt] {$z^*_{2m+1}$};
		\filldraw (5, 4.25) circle (1.5pt) node[below=2pt] {$z^*_{2m+2}$};
		\filldraw (5, 6.5) circle (1.5pt) node[below=2pt] {$z^*_{2m+3}$};
		\node at (5, 8.5) {$\vdots$};
		\filldraw (5, 10.75) circle (1.5pt) node[below=2pt] {$z^*_{3m}$};
		
		% --- Center Gap (Column 4) ---
		% Parallel horizontal dots
		\node at (7, 2.5) {$\cdots$};
		\node at (7, 6) {$\cdots$};
		\node at (7, 9.5) {$\cdots$};
		
		% --- Column m-1 (Second to last) ---
		\draw (8,3.5) -- (10,3.5);
		\draw (8,6.5) -- (10,6.5);
		\draw[dashed] (8,8.5) -- (10,8.5);
		\draw[dashed] (8,10) -- (10,10);
		\filldraw (9, 1.75) circle (1.5pt) node[below=2pt] {$z^*_{m^2-2m+1}$};
		\filldraw (9, 5) circle (1.5pt) node[below=2pt] {$z^*_{m^2-2m+2}$};
		\filldraw (9, 7.5) circle (1.5pt) node[below=2pt] {$z^*_{m^2-2m+3}$};
		\node at (9, 9.35) {$\vdots$};
		\filldraw (9, 11) circle (1.5pt) node[below=2pt] {$z^*_{m^2-m}$};
		
		% --- Column m (Last) (Fixed: Added parallel dashed line & dot) ---
		\draw (10,3.5) -- (12,3.5);
		\draw (10,6) -- (12,6);
		\draw[dashed] (10,8) -- (12,8);
		\draw[dashed] (10,10) -- (12,10);
		\filldraw (11, 1.75) circle (1.5pt) node[below=2pt] {$z^*_{m^2-m+1}$};
		\filldraw (11, 4.75) circle (1.5pt) node[below=2pt] {$z^*_{m^2-m+2}$};
		\filldraw (11, 7) circle (1.5pt) node[below=2pt] {$z^*_{m^2-m+3}$};
		\node at (11, 9) {$\vdots$};
		\filldraw (11, 11) circle (1.5pt) node[below=2pt] {$z^*_n$};
		
		% --- Corner Labels ---
		\node[below left] at (0,0) {$(0,0)$};
		\node[below right] at (12,0) {$(L,0)$};
		\node[above right] at (12,12) {$(L,L)$};
		\node[above left] at (0,12) {$(0,L)$};
		
		% --- Bottom Domain Label ---
		\node[below, yshift=-10pt] at (6,0) {\Large $D$};
		
	\end{tikzpicture}
	\caption{Dictionary-ordered partitioning of square $D$ at the $n=m^2$-th stage.}
	\label{fig:partition_D}
\end{figure}

\noindent 
{\it Case-II: }Next we construct $n$ points when $n$ is not a perfect square. That is, there exists $m\in \N$ such that $m^2<n<(m+1)^2$. There exists $1\le k\le 2m$ such that 
\[
n=m^2+k.
\]
In this case, we choose $0=x_0,x_1,\ldots,x_m<L$ such that 
\[
\nu(R_i)=\frac{1}{\sqrt n}, \mbox{ where $R_i=[x_{i-1},x_i]\times [0,L]$ for } i=1,\ldots,m.
\]
Then, for each $i=1,\ldots, m$, choose $0=y_{i0},y_{i1},\ldots, y_{im}<L$ such that 
\[
\nu(R_{ij})=\frac{1}{n}, \mbox{ where }	R_{ij}=[x_{i-1},x_i]\times [y_{j(j-1)},y_{ij}], \mbox{ for $j=1,\ldots,m$}.
\]
Let $R=\cup_{i,j=1}^mR_{ij}$ and $R'=D\backslash R$. Observe that $\nu(R')=k/n$. Then we divide the region $R'$ from right-bottom corner to left-top corner of $D$ in $k$ regions $R_{1}',\ldots, R_k'$ consecutively, as shown in Figure 2, such that 
\[
\nu(R_i')=\frac{1}{n}, \mbox{ for } i=1,\ldots, k.
\]
Thus we have $n$ regions $\{R_{ij}\suchthat 1\le i,j\le m\}\cup \{R_i'\suchthat 1\le i\le k\}$ such that the measure of each region is $1/n$. We choose first $m^2$ points $\{z_1^*,\ldots,z_{m^2}^*\}$ as before. The rest of the $k$ points are chosen from the $k$ regions $\{R_i'\suchthat 1\le i\le k\}$.  We choose $z_{m^2+i}^*\in R_i'\cap \{\{x=L\}\cup \{y=L\}\}$ (from the right or top side of $D$ on $R_i'$) such that $R_i'$ is divided in two equal (measure) parts.
Thus we have $n$ many points $\{z_1^*,\ldots,z_n^*\}$ from $D$.
Hence we have $n$ points in both cases.

%%%	\begin{figure}[h]
	%%	\centering
	%%%%\includegraphics[width=0.5\linewidth]{1000388945.jpg}
	%%%%%	\label{fig:placeholder}
	%%%%%%%%	\end{figure}
%\end{frame}
\begin{figure}[htbp]
	\centering
	\begin{tikzpicture}[scale=0.9, thick]
		% Outer boundary of the domain D
		\draw (0,0) rectangle (12,12);
		
		% Corner Labels
	\node[below left] at (0,0) {$(0,0)$};
	\node[below right] at (12,0) {$(L,0)$};
	\node[above right] at (12,12) {$(L,L)$};
	\node[above left] at (0,12) {$(0,L)$};
	
		% Bottom Domain Label
		\node[below, yshift=-10pt] at (6,0) {\Large $D$};
		
		% ==========================================
		% --- VERTICAL BOUNDARIES ---
		% ==========================================
		
		% Col 1 | Col 2 
		% Solid line covers the left side of z_2m^*, then becomes dashed above it
		\draw (2,0) -- (2,10.5);         
		
		\draw (4,0) -- (4,10.5);        % Col 2 | Col 3
		
		% Col 3 | Gap
		% Dashed line on the right side of z_3m^* up to the roof
		\draw[dashed] (5.5,0) -- (5.5,9.5); 
		
		\draw[dashed] (7,0) -- (7,9.5);     % Gap | Col 5
		\draw (8.75,0) -- (8.75,9.5);       % Col 5 | Col 6 
		\draw (10.5,0) -- (10.5,10.5);      % Col 6 | Right Free Space 
		
		% ==========================================
		% --- ROOFS OF THE COLUMNS ---
		% ==========================================
		
		\draw (0,9.5) -- (2,9.5);         % Roof of Col 1
		\draw (2,10.5) -- (4,10.5);       % Roof of Col 2
		\draw (4,9.5) -- (5.5,9.5);       % Roof of Col 3
		
		% Roof of the Gap 
		\draw (5.5,9.5) -- (7,9.5);     
		
		\draw (7,9.5) -- (8.75,9.5);      % Roof of Col 5
		\draw (8.75,10.5) -- (10.5,10.5); % Roof of Col 6
		
		% ==========================================
		% --- UPPER FREE SPACE DIVIDERS ---
		% ==========================================
		
		% Divider extending above Col 1 (Right side of z_n^* region)
		\draw[dashed] (2,10.5) -- (2,12);
		
		% Dividers extending above Col 3 
		\draw[dashed] (4,10.5) -- (4,12); 
		\draw (5.5,9.5) -- (5.5,12);      % Made solid vertical line where z_{m^2+4} lies
		
		% Dividers extending above Col 5 
		\draw (7,9.5) -- (7,12);
		\draw (8.75,9.5) -- (8.75,12);
		
		% Divider extending above Col 6 (Made solid vertical line instead of dots)
		\draw (10.5,10.5) -- (10.5,12);
		
		% Right side free space horizontal dividers
		\draw (10.5, 3.5) -- (12, 3.5);
		\draw (10.5, 7.5) -- (12, 7.5);
		
		% ==========================================
		% --- POINTS & REGIONS WITHIN COLUMNS ---
		% ==========================================
		
		% Column 1
		\draw (0,3) -- (2,3);
		\draw (0,5.5) -- (2,5.5);
		\draw[dashed, thin] (0,7) -- (2,7);
		\draw[dashed, thin] (0,8) -- (2,8);
		\filldraw (1, 1.5) circle (1.5pt) node[above=2pt] {$z^*_1$};
		\filldraw (1, 4.25) circle (1.5pt) node[above=2pt] {$z^*_2$};
		\filldraw (1, 6.25) circle (1.5pt) node[above=2pt] {$z^*_3$};
		\node at (1, 7.6) {$\vdots$};
		\filldraw (1, 8.75) circle (1.5pt) node[above=2pt] {$z^*_m$};
		
		% Column 2
		\draw (2,3.5) -- (4,3.5);
		\draw (2,6) -- (4,6);
		\draw (2,7.5) -- (4,7.5);
		\draw[dashed, thin] (2,8.5) -- (4,8.5); 
		\draw[dashed, thin] (2,9.5) -- (4,9.5); 
		\filldraw (3, 1.75) circle (1.5pt) node[above=2pt] {$z^*_{m+1}$};
		\filldraw (3, 4.75) circle (1.5pt) node[above=2pt] {$z^*_{m+2}$};
		\filldraw (3, 6.75) circle (1.5pt) node[above=2pt] {$z^*_{m+3}$};
		\node at (3, 9.15) {$\vdots$};
		\filldraw (3, 9.8) circle (1.5pt) node[above=2pt] {$z^*_{2m}$}; 
		
		% Column 3
		\draw (4,3) -- (5.5,3);
		\draw (4,5) -- (5.5,5);
		\draw[dashed, thin] (4,6.5) -- (5.5,6.5); 
		\draw[dashed, thin] (4,7.5) -- (5.5,7.5); 
		\filldraw (4.75, 1.5) circle (1.5pt) node[above=2pt] {$z^*_{2m+1}$};
		\filldraw (4.75, 4) circle (1.5pt) node[above=2pt] {$z^*_{2m+2}$};
		\node at (4.75, 7.15) {$\vdots$};
		\filldraw (4.75, 8.5) circle (1.5pt) node[above=2pt] {$z^*_{3m}$};
		
		% Central Gap
		\node at (6.25, 4.5) {\Large $\cdots$};
		
		% Column 5 (Col m-1)
		\draw (7,2.5) -- (8.75,2.5);
		\draw (7,4.5) -- (8.75,4.5);
		\draw[thin] (7,6.5) -- (8.75,6.5); % Removed bold line, made thin
		\draw[dashed, very thin] (7,7.5) -- (8.75,7.5); 
		\draw[dashed, very thin] (7,8.5) -- (8.75,8.5); 
		\filldraw (7.875, 1.25) circle (1.5pt) node[above=2pt] {\footnotesize $z^*_{m^2-2m+1}$};
		\filldraw (7.875, 3.5) circle (1.5pt) node[above=2pt] {\footnotesize $z^*_{m^2-2m+2}$};
		\filldraw (7.875, 5.5) circle (1.5pt) node[above=2pt] {\footnotesize $z^*_{m^2-2m+3}$};
		\filldraw (7.875, 6.9) circle (1.5pt) node[above=2pt] {\footnotesize $z^*_{m^2-2m+4}$}; 
		\node at (7.875, 8.0) {$\vdots$};
		\filldraw (7.875, 8.8) circle (1.5pt) node[above=2pt] {\footnotesize $z^*_{m^2-m}$}; 
		
		% Column 6 (Col m)
		\draw (8.75,3.5) -- (10.5,3.5);
		\draw (8.75,6) -- (10.5,6);
		\draw[dashed, thin] (8.75,7.5) -- (10.5,7.5); 
		\draw[dashed, thin] (8.75,8.5) -- (10.5,8.5);
		\filldraw (9.625, 1.75) circle (1.5pt) node[above=2pt] {\footnotesize $z^*_{m^2-m+1}$};
		\filldraw (9.625, 4.75) circle (1.5pt) node[above=2pt] {\footnotesize $z^*_{m^2-m+2}$};
		\node at (9.625, 8.15) {$\vdots$};
		\filldraw (9.625, 9.5) circle (1.5pt) node[above=2pt] {\footnotesize $z^*_{m^2}$};
		
		% ==========================================
		% --- POINTS ON BOUNDARIES ---
		% ==========================================
		
		% Points on the TOP Boundary (z_4 to z_3 line)
		\filldraw (1, 12) circle (1.5pt) node[above=2pt] {$z^*_n$};
		\filldraw (4.75, 12) circle (1.5pt) node[above=2pt] {$z^*_{m^2+7}$}; 
		\filldraw (6.25, 12) circle (1.5pt) node[above=2pt] {$z^*_{m^2+6}$};
		\filldraw (7.875, 12) circle (1.5pt) node[above=2pt] {$z^*_{m^2+5}$};
		\filldraw (9.625, 12) circle (1.5pt) node[above=2pt] {$z^*_{m^2+4}$}; 
		\filldraw (11.25, 12) circle (1.5pt) node[above=2pt] {$z^*_{m^2+3}$}; 
		
		% Points on the RIGHT Boundary (z_2 to z_3 line)
		\filldraw (12, 1.75) circle (1.5pt) node[left=4pt] {$z^*_{m^2+1}$};
		\filldraw (12, 5.5) circle (1.5pt) node[left=4pt] {$z^*_{m^2+2}$};
		
	\end{tikzpicture}
	\caption{Grid partition of $D$ divided into $m^2$ perfect square portions and remaining non-square portions mapped to the upper and right boundaries.}
	\label{fig:partition_final_v6}
\end{figure}
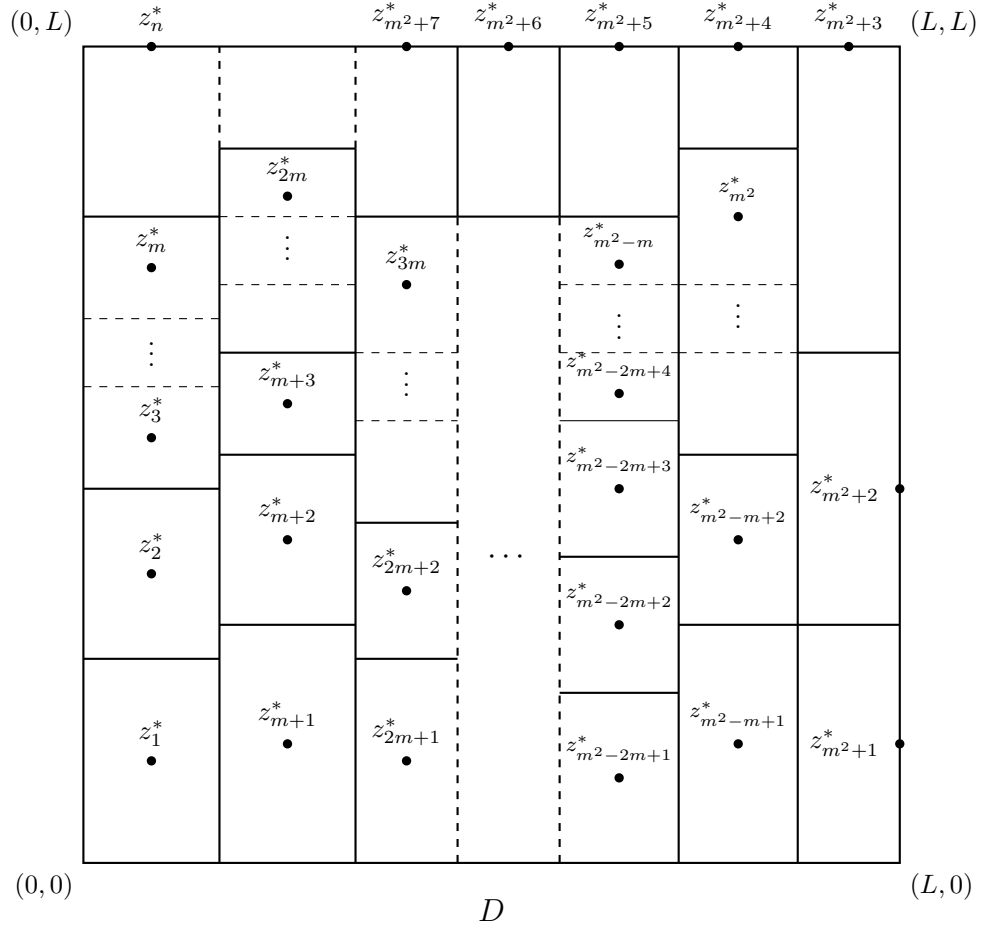
\noindent {\it Proof of (1):}
Let $\nu_n=\frac{1}{n}\sum_{i=1}^{n}\delta_{z_i^*}$ and $f$ be a bounded continuous function. Then we show that the following holds
\[
\l|{\int_{D}f d\nu_{n}-\int_{D}f d\nu}\r|\to 0, \text{ as  } n\to \infty.
\]
Note that  $f$ is uniformly continuous on $D$, as $D$ is compact. Moreover, the length of the sub-rectangle $R_{i,n}$ goes to zero as $n$ tends to $\infty$. Therefore, for given $\epsilon>0$,
\[
\sup_{i\in [n]}\sup_{z_i\in R_{i,n}}|f(z_i^*)-f(z_i)|<\epsilon
\]
for large $n$. 	Which implies that, for $i=1,\ldots,m^2$, 
\begin{align}\label{eqn:uniform}
	\l|\frac{1}{n}f(z_i^*)-\int_{R_{i,n}}fd\nu \r|&\le \int_{R_{i,n}}\l|f(z_i^*)-f(z_i)\r|d\nu(z)\leq \frac{\epsilon}{n}.
\end{align}
On the other hand, since $f$ is uniformly continuous on $D$, we have 
\begin{align}\label{eqn:crudebound}
	\l|\frac{1}{n}f(z_i^*)-\int_{R_{i,n}}fd\nu \r|\le \frac{2M}{n}, \mbox{ for all } i=1,\ldots, n.
\end{align}
Therefore, by \eqref{eqn:uniform} and \eqref{eqn:crudebound} for $n=m^2+k$, we have 
\begin{align*}
	\l|{\int_{D}f d\nu_{n}-\int_{D}f d\nu}\r|&\leq \sum_{i=1}^{m^2}\l|\frac{1}{n}f(z_i^*)-\int_{R_{i,n}}fd\nu\r|+\sum_{i=m^2+1}^{n}\l|\frac{1}{n}f(z_i^*)-\int_{R_{i,n}}fd\nu\r|
	\\&\le \frac{m^2\epsilon}{n}+\frac{2kM}{n}\le 2\epsilon,
\end{align*}
for all large $n$. Which gives the result.

\vspace{.1cm}
\noindent {\it Proof of part (2)}. \underline{\it Case-I}. Suppose $n=m^2$. Let $a_i$ be  the side of  rectangles $\bar R_{i}$. Since the measure $\nu$ has density $g_\nu$ with $1/C\le g_\nu\le C$. Then we have
\[
\frac{1}{C}m(\bar R_{i})\leq \int_{\bar R_{i}}g_{\nu}dm\leq Cm(\bar R_{i}).
\]
Since $m(\bar R_i)=a_iL$ and $\nu(\bar R_i)=1/\sqrt{n}$, hence we have
\begin{align}\label{eqn:a_i}
	\frac{1}{CL\sqrt{n}}\leq a_i\leq \frac{C}{L\sqrt{n}}, \mbox{  for $i=1,\ldots,m$.}
\end{align}
Let $b_i$ be the side length of the sub-rectangle  $R_{i, n}$, for $i=1,\ldots, m^2$. Then, by \eqref{eqn:a_i} and the same arguments as before, we have
\begin{align*}
	\frac{L}{\sqrt{n}C^2}\leq b_i\leq \frac{L}{\sqrt{n}}, \mbox{ for } i=1,\ldots,m^2.
\end{align*}
Thus the side lengths of the rectangle $R_{i,n}$ is $O(1/\sqrt{n})$, for $i=1,\ldots,m^2$.

\vspace{.1cm}
\noindent \underline{\it Case-II.} Suppose $n=m^2+k$ for some $m\in N$, where $0\le k\le 2m$. Consider the rectangle $\bar R_{m+1}=D\backslash \cup_{i=1}^m\bar R_i$.	Then, as $m=\sqrt{n-k}$,
\begin{align*}
	\nu(\bar R_{m+1})=1-\frac{m}{\sqrt n}=\frac{\sqrt{n}-\sqrt{n-k}}{\sqrt{n}}\ge \frac{k}{2n}.
\end{align*}
Let $a_{m+1}$ be the width of the rectangle $\bar R_{m+1}$. Then  as before we have 
\begin{align*}
	a_{m+1}\ge \frac{k}{2CLn}.
\end{align*}
Let $R_i'=\bar R_i\backslash \cup_{j=1}^mR_{i,j}$ for $i=1,\ldots, m$, where $R_{i,j}$ are as defined in \eqref{def:Rij}. Then
\[
\nu(R_i')=\frac{1}{\sqrt n}-\frac{m}{n}\ge \frac{k}{n^{3/2}}.
\]
We know that one of the side length is $a_i$. Suppose the length of the other side is $a_i'$. Therefore, using $Cm(R_i')\ge \nu(R_i')$, we get
\[
a_i'\ge \frac{kL}{C^2n}.
\]
Thus, by the construction of points $\{z_1^*,\ldots,z_n^*\}$, we have 
\[
|z_i^*-z_j^*|\ge \frac{1}{2\sqrt{ Cn}}, \mbox{ for all } 1\le i,j\le n,
\]
as the distance between any two points is atleast the half of the side length of the rectangles. Hence the result.
\end{proof}

\begin{proof}[Proof of \Cref{lem:energy}] Recall $\nu$ and $z_1^*,\ldots, z_n^*$ are as in \Cref{lem:construction of points}.
The proof is considered in following two cases.

\noindent \underline{\it Case-I.} (For perfect square). Suppose that $n=m^2$ for some $m\in \N$. Let
\[
B=\{(i,j)\suchthat 1\leq i,j \leq m, |z_{i}^*-z_{j}^*|>{1}/{\sqrt{m}}\}.
\]
Then consider the following 
\begin{align*}
	&\iint \limits_{z,w\in D} \log|z-w| d{\nu}(z)d{\nu}(w)
	\leq \sum_{i,j=1}^{m}\sup_{z\in R_{i,n},w\in R_{j,n}} \log|z-w|\iint\limits_{R_{i,n}\times R_{j,n}}d{\nu}(z)d{\nu}(w).
	%	\\&\leq\sum_{(i,j)\in B^c}\iint_{R_{i,n},R_{j,n}}\log[(z_{i,n}-z_{j,n})(1+c)]d\Tilde{\nu}d\Tilde{\nu}+\sum_{(i,j)\in B}\iint_{R_{i,n},R_{j,n}}\log[|z_{i,n}-z_{j,n}|(1+{\frac{1}{\sqrt{m}}})]d\Tilde{\nu}d\Tilde{\nu}
	%	\\& \leq \frac{1}{n^2}\sum_{(i,j)\in B^c} \log{|z_{i,n}-z_{j,n}|}+\frac{1}{n^2}|B^c|\log(1+c)+\frac{1}{n^2}\sum_{(i,j)\in B}\log|z_{i,n}-z_{j,n}|+\frac{1}{n^2}\log(1+\frac{1}{\sqrt{m}})\\&=\frac{1}{m^4} \sum_{(i,j)\in B^c}\log|z_{i,n}-z_{j,n}|+\frac{1}{m^4}|B^c|\log(1+c)+\sum_{(i,j)\in B}\log|z_{i,n}-z_{j,n}|+\log(1+\frac{1}{\sqrt{m}})
\end{align*}
Let $z\in R_{i,n}$ and $w\in R_{j,n}$. Suppose $(i,j)\in B$, then 
\begin{align}\label{eqn:inB}
	|z-w|\le |z_i^*-z_j^*|+2\sqrt{C}/m\le |z_i^*-z_j^*|(1+2\sqrt{C/m}).
\end{align}
Since the side lenght of $R_{i,n} $ are atleast $1/m\sqrt{C}$, then we have $|z_i^*-z_j^*|\ge 1/2m\sqrt{C}$. Which implies that, for $z\in R_{i,n}$, $w\in R_{j,n}$ and $(i,j)\in B^c$,
\begin{align}\label{eqn:inBc}
	|z-w|\le |z_i^*-z_j^*|+2\sqrt{C}/m\le |z_i^*-z_j^*|(1+4C).
\end{align}
Then, by \eqref{eqn:inB}, \eqref{eqn:inBc}  and the fact that $\nu(R_{i,n})=1/n$ for all $i=1,\ldots, n$, we have
\begin{align*}
	&	\iint \limits_{z,w\in D} \log|z-w| d{\nu}(z)d{\nu}(w)
	\\\le& \frac{1}{n^2}\sum_{(i,j)\in B^c} \log( |z_i^*-z_j^*|(1+4C))
	+\frac{1}{n^2}\sum_{(i,j)\in B}\log( |z_i^*-z_j^*|(1+2\sqrt{C/m}))
	\\=&\frac{1}{n^2}\sum_{(i,j)} \log( |z_i^*-z_j^*|)+\frac{|B^c|}{n^2} \log((1+4C))+\frac{|B|}{n^2} \log((1+2\sqrt{C/m})).
\end{align*}
Observe that, $|B^c|/n^2\to 0$ and $\log((1+2\sqrt{C/m}))\to 0$ as $n\to \infty$. Therefore
\begin{align}
	\iint \limits_{z,w\in D} \log|z-w| d{\nu}(z)d{\nu}(w)\le \liminf_{n \to \infty}\frac{1}{n^2}\sum_{(i,j)} \log( |z_i^*-z_j^*|).
\end{align}

%%%%%%%%%%%%%%%%%%%%%%%%%%%%%%%%%%%%%%%%%%%%%%%%%%%%%%%%%%\end{proof}
%%%%%%%%%%%%%%%%%%%%%%%%%%%%%%%%%%%%%%%%%%%%%%%%%%%%%%%%%%%%%%%%%%%%%%%%%%%%%%%%%%%%%%%%%%%%%%%%%%%%%%%%%%%%%%%%%%%%%%%%%%%%%%%%%%%%%%%%%%%%%%%%%%%%%%%%%%%%%%%%%%%%%%%%%%%%%%%%%%%%%%%%%%%%%%%%%%%%%%%%%%%%%%%%%%%%%%%%%%%%%%%%%%%%%%%%%%%%%%%%%%%%%%%%%%%%%%%%%%%%%%%%%%%%%%%%%%%%%%%%%%%%%%%%%%%%%%%%%%%%%%%%%%%%%%%%%%%%%%%%%%%%%%%%%%%%%%%%%%%%%

\noindent \underline{\it Case-II.} (For general n). Suppose $m^2<n<(m+1)^2$ for some $m\in \mathbb N$. Then there exists  $1\le k\le 2m$ such that $n=m^2+k$. Therefore we have
\begin{align}\label{eqn:together}
	\iint \limits_{z,w\in D} \log|z-w| d{\nu}(z)d{\nu}(w)&=\sum_{1\leq i,j \leq m^2}\iint_{R_{i,n}\times R_{j,n}}\log|z-w| d\nu(z) d\nu(w)\nonumber
	\\&
	%\leq \sum_{i,j}\iint_{R_{i,n},R_{j,n}}\log|z-w|d\nu d\nu\\&
	+2\sum_{1\leq i\leq n,m^2< j\leq n}\iint_{R_{i,n}\times R_{j,n}}\log|z-w| d\nu(z) d\nu(w).
\end{align}
Observe that by \Cref{lem:construction of points} we have $|z_i^*-z_j^*|\ge 1/2\sqrt{C}m$ for all $i\neq j$.  Suppose $L$ be the diameter of $D$. Then, for $z\in R_{i,n}$ and $w\in R_{j,n}$, we have
\begin{align*}
	|z-w|\le 2L\sqrt{C}m|z_i^*-z_j^*|, \mbox{ for all } i\neq j.
\end{align*}
Which implies that, as $\nu(R_{i,n})=1/n$ for $1\le i\le n$,  
\begin{align}\label{eqn:error}
	&\sum_{1\leq i\leq n,m^2< j\leq n}\iint_{R_{i,n}\times R_{j,n}}\log|z-w| d\nu(z) d\nu(w)
	\\\leq & \frac{1}{n^2}\sum_{1\leq i\leq n,m^2< j\leq n}\log(2L\sqrt{C}m|z_i^*-z_j^*|)\nonumber
	\\
	=&\frac{1}{n^2}\sum_{1\leq i\leq n,m^2< j\leq n}\log(2L\sqrt{C}m)+\frac{1}{n^2}\sum_{1\leq i\leq n,m^2< j\leq n}\log(|z_i^*-z_j^*|).\nonumber
\end{align}
On the other hand by the same calculation as in \underline{Case-I} we get
\begin{align}\label{eqn:main}
	&\sum_{1\leq i,j \leq m^2}\iint_{R_{i,n}\times R_{j,n}}\log|z-w| d\nu(z) d\nu(w)
	\\\le& \frac{1}{n^2}\sum_{1\le i,j\le m^2} \log( |z_i^*-z_j^*|)+\frac{|B^c|}{n^2} \log((1+4C))+\frac{|B|}{n^2} \log((1+2\sqrt{C/m})).\nonumber
\end{align}
Therefore by \eqref{eqn:error} and \eqref{eqn:main} from \eqref{eqn:together} we get 
\begin{align*}
	\iint \limits_{z,w\in D} \log|z-w| d{\nu}(z)d{\nu}(w)\le  \liminf_{n \to \infty}\frac{1}{n^2}\sum_{1\le i,j\le n} \log( |z_i^*-z_j^*|).
\end{align*}
Hence the result.
\end{proof}

\begin{proof}[Proof of \Cref{lem:openset}]
Let $z_1^*,\ldots,z_n^*$ be as in  \Cref{lem:construction of points}. Define,
\[
S=\l\{\frac{1}{n}\sum_{k=1}^{n}\delta_{z_k}\suchthat |z_k-z_{k}^*|\le \frac{1}{n^3} \mbox{ for all } k=1,2,...,n\r\}.
\]
Let $\mu_n\in S$. Then $d(\mu_n, \mu_n^*)\le 1/n^3$, where 
\[
\mu_n^*=\frac{1}{n}\sum_{k=1}^{n}\delta_{z_k^*}.
\] 
On the other hand, by \Cref{lem:construction of points}, we have $\mu_n^*$ converges to $\nu$ weakly as $n\to \infty$. Which implies that $\mu_n\in S$ for large $n$. Hence $S\subset B(\nu,\delta)$ for large $n$. Thus
\begin{align*}
	\bar\P^{\mathbb C}_{n,\beta}(L_n \in B(\nu,\delta))&\geq \bar\P^{\mathbb C}_{n,\beta}(L_n\in S)\\&
	=\int\limits_{|z_1-z_{1}^*|<\frac{1}{n^3}}\cdots\int\limits_{|z_n-z_{n}^*|<\frac{1}{n^3}}\prod_{i<j}|z_{i}-z_{j}|^{\beta}e^{-n\sum\limits_{i=1}^{n}V(z_i)}\prod_{i=1}^{n}dm(z_i).
\end{align*}
%So here we've chosen $\epsilon_n=\frac{1}{n^3}$ very small, where $k\in\mathbb{N}$ and it's depending on n (later, accordingly with the need, we'll choose $\epsilon_n$). 
%now, few estimates are required to have a bound on the above equation. 
Note that  $V$ is  uniformly continuous function on $ D$, as $V$ is continuous on $\mathbb F$.  Then
\begin{align}\label{eqn:v}
	|V(z_k)-V(z_k^*)|\le 1, \mbox{ for all $k=1,\ldots,n$, for large $n$}.
\end{align}
Again,  by \Cref{lem:construction of points}, we have $|z_i^*-z_j^*|\ge 1/2\sqrt{Cn}\ge \frac{1}{n}$, for all $i,j \in {1,2,\ldots,n}$. 
Therefore, if $|z_i-z_i^*|<1/n^3$ and $|z_j-z_j^*|<1/n^3$ for all $1\le i,j\le n$, 
\begin{align}\label{eqn:separation}
	|z_{i}-z_{j}|\geq |z_i^*-z_j^*|-\frac{2}{n^3}\geq |z_i^*-z_j^*|(1-\frac{2}{n^2}), 
	\;\;\mbox{ for $1\le i,j\le n$}.
\end{align}
Thus \eqref{eqn:v} and \eqref{eqn:separation}  imply that 
\begin{align*}
	\bar\P^{\mathbb C}_{n,\beta}(L_n\in S)
	&\geq \idotsint\limits_{B(z_1,n^{-3})\; B(z_n,n^{-3})}\prod_{i<j}||z_i^*-z_j^*|(1-2n^{-2})|^{\beta}e^{-n\sum\limits_{i=1}^{n}V(z_i^*)}e^{-n}\prod_{i=1}^{n}dm(z_i)\\&
	=\prod_{i<j}|z_i^*-z_j^*|^{\beta}\l(1-2n^{-2}\r)^{\frac{\beta n(n-1)}{4}}e^{-n\sum_{i=1}^{n}V(z_{i}^*)}e^{-n}\prod_{i=1}^n|B(z_i,n^{-3})|.
	\\&
	=\prod_{i<j}|z_i^*-z_j^*|^{\beta}\l(1- 2n^{-2}\r)^{\frac{\beta n(n-1)}{4}}e^{-n\sum_{i=1}^{n}V(z_{i}^*)}e^{-n}(\pi n^{-6})^n.
\end{align*}
Since $\log(\cdot)$ is an increasing function, taking $\log $ and limit in both side, we get
\begin{align*}
	\liminf_{n\to \infty} \frac{1}{n^2}\log\bar \P^{\mathbb C}_{n,\beta}(L_n\in S)
	%	\geq\liminf_{n\to\infty}\{\frac{\beta}{2n^2}\sum_{i<j}\log|z_i^*-z_j^*|+\frac{\beta n(n+1)}{2n^2}\log(|1-\frac{1}{n^k}|)-\frac{1}{n}\sum_{i=1}^{n}(V(z_i^*)-\delta_n)\\&
	%	\geq \liminf_{n\to\infty}\{\frac{\beta}{2n^2}\sum_{i<j}\log|z_i^*-z_j^*|\}+\liminf_{n\to\infty}\{\frac{\beta n(n+1)}{2n^2}\log(|1-\frac{1}{n^k}|)\}\\&-\liminf_{n\to\infty}\{\frac{1}{n}\sum_{1}^{n}V(z_i^*)\}-\liminf_{n\to\infty}\{\epsilon_n\}\\&
	&=\liminf_{n\to\infty}\frac{\beta}{2n^2}\sum_{i\neq j}\log|z_i^*-z_j^*|-\liminf_{n\to\infty}\frac{1}{n}\sum_{1}^{n}V(z_i^*)
	\\&\ge \frac{\beta}{2}\iint \log |z-w|d\nu(z)\nu(w)-\int V(z)d\nu(z).
	%	\\&=-I_\beta(\nu).
\end{align*}
The second line  follows from \Cref{lem:energy} and the fact that $V$ is bounded continuous on $D$ and $\mu_n^*$ converges to $\nu$ weakly as $n\to \infty$. Thus, by \Cref{fact:3} we get
\[
\liminf_{n\to \infty} \frac{1}{n^2}\log\P^{\mathbb C}_{n,\beta}(L_n\in B(\nu;\delta))\ge -I_\beta(\nu).
\]
Hence the result.
\end{proof}

%%%%%%%%%%%%%%%%%%%%%%%now the proof will be completed if we use the following lemma

\bibliographystyle{abbrv}
\bibliography{references.bib}

\begin{thebibliography}{10}

\bibitem{anderson2010introduction}
G.~W. Anderson, A.~Guionnet, and O.~Zeitouni.
\newblock {\em An introduction to random matrices}.
\newblock Number 118. Cambridge university press, 2010.

\bibitem{arous1997large}
G.~B. Arous and A.~Guionnet.
\newblock Large deviations for wigner's law and voiculescu's non-commutative
  entropy.
\newblock {\em Probability theory and related fields}, 108:517--542, 1997.

\bibitem{MR3492936}
F.~Augeri.
\newblock Large deviations principle for the largest eigenvalue of {W}igner
  matrices without {G}aussian tails.
\newblock {\em Electron. J. Probab.}, 21:Paper No. 32, 49, 2016.

\bibitem{ben1998large}
G.~Ben~Arous and O.~Zeitouni.
\newblock Large deviations from the circular law.
\newblock {\em ESAIM: Probability and Statistics}, 2:123--134, 1998.

\bibitem{billingsley2013convergence}
P.~Billingsley.
\newblock {\em Convergence of probability measures}.
\newblock John Wiley \& Sons, 2013.

\bibitem{bloom2014large}
T.~Bloom.
\newblock Large deviation for outlying coordinates in $\beta$ ensembles.
\newblock {\em Journal of Approximation Theory}, 180:1--20, 2014.

\bibitem{MR3265172}
C.~Bordenave and P.~Caputo.
\newblock A large deviation principle for {W}igner matrices without {G}aussian
  tails.
\newblock {\em Ann. Probab.}, 42(6):2454--2496, 2014.

\bibitem{MR3262506}
D.~Chafa\"i, N.~Gozlan, and P.-A. Zitt.
\newblock First-order global asymptotics for confined particles with singular
  pair repulsion.
\newblock {\em Ann. Appl. Probab.}, 24(6):2371--2413, 2014.

\bibitem{cook2023full}
N.~A. Cook, R.~Ducatez, and A.~Guionnet.
\newblock Full large deviation principles for the largest eigenvalue of
  sub-gaussian wigner matrices.
\newblock {\em arXiv preprint arXiv:2302.14823}, 2023.

\bibitem{cramer1994nouveau}
H.~Cram{\'e}r.
\newblock Sur un nouveau th{\'e}oreme-limite de la th{\'e}orie des
  probabilit{\'e}s.
\newblock In {\em Collected Works II}, pages 895--913. Springer, 1994.

\bibitem{dembo2009large}
A.~Dembo and O.~Zeitouni.
\newblock {\em Large deviations techniques and applications}, volume~38.
\newblock Springer Science \& Business Media, 2009.

\bibitem{deuschel1989large}
J.~Deuschel and D.~Stroock.
\newblock Large deviations, acad.
\newblock {\em Press, Boston}, 1989.

\bibitem{deuschel2001large}
J.-D. Deuschel and D.~W. Stroock.
\newblock {\em Large deviations}, volume 342.
\newblock American Mathematical Soc., 2001.

\bibitem{MR1936554}
I.~Dumitriu and A.~Edelman.
\newblock Matrix models for beta ensembles.
\newblock {\em J. Math. Phys.}, 43(11):5830--5847, 2002.

\bibitem{dyson1962statistical}
F.~J. Dyson.
\newblock Statistical theory of the energy levels of complex systems. i.
\newblock {\em Journal of Mathematical Physics}, 3(1):140--156, 1962.

\bibitem{federer2014geometric}
H.~Federer.
\newblock {\em Geometric measure theory}.
\newblock Springer, 2014.

\bibitem{forrester2010log}
P.~J. Forrester.
\newblock {\em Log-gases and random matrices (LMS-34)}.
\newblock Princeton university press, 2010.

\bibitem{ginibre1965statistical}
J.~Ginibre.
\newblock Statistical ensembles of complex, quaternion, and real matrices.
\newblock {\em Journal of Mathematical Physics}, 6(3):440--449, 1965.

\bibitem{MR2095566}
A.~Guionnet.
\newblock Large deviations and stochastic calculus for large random matrices.
\newblock {\em Probab. Surv.}, 1:72--172, 2004.

\bibitem{MR2926763}
A.~Hardy.
\newblock A note on large deviations for 2{D} {C}oulomb gas with weakly
  confining potential.
\newblock {\em Electron. Commun. Probab.}, 17:no. 19, 12, 2012.

\bibitem{hedenmalm2013coulomb}
H.~Hedenmalm and N.~Makarov.
\newblock Coulomb gas ensembles and laplacian growth.
\newblock {\em Proceedings of the London Mathematical Society},
  106(4):859--907, 2013.

\bibitem{Holcomb2017OvercrowdingAF}
D.~Holcomb and B.~Valk{\'o}.
\newblock Overcrowding asymptotics for the sine(beta) process.
\newblock {\em Annales De L Institut Henri Poincare-probabilites Et
  Statistiques}, 53:1181--1195, 2017.

\bibitem{husson2022large}
J.~Husson.
\newblock Large deviations for the largest eigenvalue of matrices with variance
  profiles.
\newblock {\em Electronic Journal of Probability}, 27:1--44, 2022.

\bibitem{kinderlehrer2000introduction}
D.~Kinderlehrer and G.~Stampacchia.
\newblock {\em An introduction to variational inequalities and their
  applications}.
\newblock SIAM, 2000.

\bibitem{krishnapur2006overcrowding}
M.~Krishnapur.
\newblock Overcrowding estimates for zeroes of planar and hyperbolic gaussian
  analytic functions.
\newblock {\em Journal of statistical physics}, 124(6):1399--1423, 2006.

\bibitem{lacroix2019intermediate}
B.~Lacroix-A-Chez-Toine, J.~A.~M. Garz{\'o}n, C.~S.~H. Calva, I.~P. Castillo,
  A.~Kundu, S.~N. Majumdar, and G.~Schehr.
\newblock Intermediate deviation regime for the full eigenvalue statistics in
  the complex ginibre ensemble.
\newblock {\em Physical Review E}, 100(1):012137, 2019.

\bibitem{majumdar2009large}
S.~N. Majumdar and M.~Vergassola.
\newblock Large deviations of the maximum eigenvalue for wishart and gaussian
  random matrices.
\newblock {\em Physical review letters}, 102(6):060601, 2009.

\bibitem{MR4235485}
B.~McKenna.
\newblock Large deviations for extreme eigenvalues of deformed {W}igner random
  matrices.
\newblock {\em Electron. J. Probab.}, 26:Paper No. 34, 37, 2021.

\bibitem{o1991capacities}
G.~L. O’brien and W.~Vervaat.
\newblock Capacities, large deviations and loglog laws.
\newblock In {\em Stable Processes and Related Topics: A Selection of Papers
  from the Mathematical Sciences Institute Workshop, January 9--13, 1990},
  pages 43--83. Springer, 1991.

\bibitem{pukhalskii1994theory}
A.~A. Pukhalskii.
\newblock On the theory of large deviations.
\newblock {\em Theory of Probability \& Its Applications}, 38(3):490--497,
  1994.

\bibitem{saff1997logarithmic}
E.~B. Saff, V.~Totik, et~al.
\newblock {\em Logarithmic potentials with external fields}, volume 316.
\newblock Springer, 1997.

\bibitem{MR3353821}
E.~Sandier and S.~Serfaty.
\newblock 2{D} {C}oulomb gases and the renormalized energy.
\newblock {\em Ann. Probab.}, 43(4):2026--2083, 2015.

\bibitem{serfaty2024lectures}
S.~Serfaty.
\newblock Lectures on coulomb and riesz gases.
\newblock {\em arXiv preprint arXiv:2407.21194}, 2024.

\bibitem{shirai2015ginibre}
T.~Shirai.
\newblock Ginibre-type point processes and their asymptotic behavior.
\newblock {\em Journal of the Mathematical Society of Japan}, 67(2):763--787,
  2015.

\bibitem{varadhan1966asymptotic}
S.~S. Varadhan.
\newblock Asymptotic probabilities and differential equations.
\newblock {\em Communications on Pure and Applied Mathematics}, 19(3):261--286,
  1966.

\end{thebibliography}

\end{document}